\documentclass{amsart}

\pdfoutput=1

\usepackage{amssymb,amsmath,amsthm,enumerate,fullpage,float}
\usepackage[usenames,dvipsnames]{color}
\usepackage{bm}
\usepackage{stmaryrd}
\usepackage{tikz}
\usetikzlibrary{calc,decorations.pathmorphing,decorations.markings,
decorations.pathreplacing,patterns,shapes,arrows}

\newtheorem{defn}{Definition}

\newtheorem{thm}{Theorem}
\newtheorem{conj}{Conjecture}

\newcommand{\eref}[1]{(\ref{#1})}

\renewcommand{\b}[1]{\bar{#1}}
\renewcommand{\leq}{\leqslant}
\renewcommand{\geq}{\geqslant}

\definecolor{pink}{RGB}{255,128,128}
\definecolor{lred}{RGB}{255,0,0}
\definecolor{dred}{RGB}{176,0,0}
\definecolor{burg}{RGB}{43,0,0}

\def\ie{{\it i.e.}\ }

\def\eg{{\it e.g.}}

\def\fs{\footnotesize}

\def\u{\scalebox{0.65}{$\blacktriangle$}}
\def\d{\scalebox{0.65}{$\blacktriangledown$}}
\def\l{\scalebox{0.65}{$\blacktriangleleft$}}
\def\r{\scalebox{0.65}{$\blacktriangleright$}}

\begin{document}

\title{Refined Cauchy and Littlewood identities, plane partitions and symmetry classes of alternating sign matrices}

\author{D.~Betea and M.~Wheeler}

\address{Laboratoire de Physique Th\'eorique et Hautes \'Energies, CNRS UMR 7589 and Universit\'e Pierre et Marie Curie (Paris 6), 4 place Jussieu, 75252 Paris cedex 05, France}
\email{betea@lpthe.jussieu.fr, mwheeler@lpthe.jussieu.fr}


\begin{abstract}
We prove and conjecture some new symmetric function identities, which equate the generating series of {\bf 1.} Plane partitions, subject to certain restrictions and weightings, and {\bf 2.} Alternating sign matrices, subject to certain symmetry properties. The left hand side of each of our identities is a simple refinement of a relevant Cauchy or Littlewood identity, allowing them to be interpreted as generating series for plane partitions. The right hand side of each identity is a partition function of the six-vertex model, on a relevant domain. These can be interpreted as generating series for alternating sign matrices, using the well known bijection with six-vertex model configurations.
\end{abstract}

\maketitle


\section{Introduction}

The theory of symmetric functions (and, in finitely many variables, symmetric polynomials) has come to play a major role in statistical mechanics in recent years, bridging combinatorics on the one side and probability and random processes on the other. For the ultimate review of the combinatorial theory, we refer the reader to \cite{mac} (see also \cite{sta}). At the lowest level in the theory, one finds the Schur polynomials. They can be defined in many different ways and appear naturally in combinatorics (as generating functions -- see \cite{sta}), representation theory (as characters of $GL(n)$, see \eg\ \cite{mac}) and classical integrable systems (as solutions of hierarchies of partial differential equations -- see \cite{jm} and references therein). Recently, many connections have also been found with probability theory and exactly solvable statistical mechanical models. These include Schur measures on partitions \cite{oko} and the Schur process of Okounkov and Reshetikhin for unboxed plane partitions \cite{or,bor}. A one-parameter generalization of Schur polynomials called Hall--Littlewood polynomials (see Chapter III of \cite{mac}) will play an important role in this paper. They also appear in the statistical mechanics of plane partitions (see \eg\ \cite{vul}) and in the theory of integrable models (see \eg\ \cite{fw, tsi} for type $A$ and \cite{vde} for type $BC$).

The purpose of this work is to relate certain infinite sum (Cauchy-like) identities involving symmetric polynomials with certain classes of plane partitions and alternating sign matrices (henceforth abbreviated as ASMs). To wit, we expand three quantities that appear as partition functions of three classes of ASMs (and are naturally symmetric in the indeterminates) in terms of certain classes of symmetric functions and connect these expansions to the theory of plane partitions. ASMs are in a simple bijection with six-vertex model configurations, and it is the latter point of view we adopt. They were introduced by Mills, Robbins and Rumsey in \cite{mrr} and we refer the reader to \cite{kup2} for more details on symmetry classes of ASMs. For another connection between symmetric polynomials (more precisely, characters of the classical groups) and the enumeration of alternating sign matrices we refer the reader to Okada \cite{oka}.

On the symmetric polynomials side, we will focus attention on Schur polynomials and their Hall--Littlewood generalization, as well as symplectic characters ($BC_n$-symmetric Laurent polynomials that are the irreducible characters of the symplectic group $Sp(2n)$) and their generalization -- the $BC_n$-symmetric Hall--Littlewood polynomials (see \eg\ \cite{ven}). On the ASM side, we look at ordinary ASMs (six-vertex model configurations with domain wall boundary conditions), U-turn ASMs (henceforth abbreviated as UASMs) and off-diagonally symmetric ASMs (henceforth, OSASMs) -- the terminology is borrowed from \cite{kup2}. We thus have three partition functions: $Z_{\rm ASM}, Z_{\rm UASM}$ and  $Z_{\rm OSASM}$ respectively. All three have determinantal/Pfaffian structure and are symmetric functions and we expand them (up to some factors) in two ways. First, we expand $Z_{\rm ASM}$ in two sets of Schur polynomials, $Z_{\rm UASM}$ in Schur polynomials and symplectic characters and $Z_{\rm OSASM}$ in (one set of) Schur polynomials. This is contained in Theorems \ref{thm1}, \ref{thm3} and \ref{thm4} respectively. Second, we expand the same quantities (up to some different factors) in terms of the appropriate $t$-generalizations (Hall--Littlewood polynomials for Schur and $BC_n$-symmetric Hall--Littlewood for symplectic characters). This is the content of Theorem \ref{thm2} (due to Warnaar \cite{war}) and Conjectures \ref{conj1} and \ref{conj2}. In a subsequent paper with P.~Zinn-Justin, we will prove Theorem \ref{thm2} and Conjecture \ref{conj2} in a unified way. We relate all of these expansions except for the $BC_n$-symmetric Hall--Littlewood one with classes of plane partitions (analogously to how one relates the usual Cauchy identity with plane partitions -- see for example \cite{or} and \cite{bor}). 

Of the identities (expansions) discussed, four are of Cauchy-type (for Schur polynomials and symplectic characters, and their respective $t$-generalizations) and two of Littlewood-type (for Schur polynomials indexed by partitions with even columns, and the relevant $t$-generalization).

We start with the classical Cauchy identity (summation formula -- see \cite{mac}) for Schur polynomials:
\begin{align}
\sum_{\lambda} 
s_{\lambda} (x_1,\dots,x_m) 
s_{\lambda} (y_1,\dots,y_n) 
= 
\prod_{i=1}^{m}
\prod_{j=1}^{n} 
\frac{1}{1-x_i y_j}. 
\end{align}
It has many interpretations. For our purposes, the left hand side is interpreted as a formal generating series for plane partitions with base contained in an $m \times n$ rectangle. The right hand side is then the partition function (a many parameter generalization of MacMahon's formula \cite{macm}). The Schur polynomials $s_{\lambda}$ have a one-parameter $t$-generalization in the Hall--Littlewood polynomials $P_{\lambda}$ (Chapter III of \cite{mac}). These also satisfy a Cauchy identity:
\begin{align} 
\label{mac_cauchy}
\sum_{\lambda} 
\prod_{i=1}^{\infty} \prod_{j=1}^{m_i(\lambda)} (1-t^j) 
P_{\lambda} (x_1,\dots,x_m;t) 
P_{\lambda} (y_1,\dots,y_n;t) 
=
\prod_{i=1}^{m}
\prod_{j=1}^{n} 
\frac{1-tx_i y_j}{1-x_i y_j}
\end{align}
where $m_i(\lambda)$ is the number of parts of $\lambda$ equal to $i$ (note the double product is finite as $\lambda$ is a finite sequence of positive integers listed in decreasing order). The Cauchy identity for Hall--Littlewood polynomials degenerates ($t=0$) to the Cauchy identity for Schur polynomials and can also be viewed as a generating series for plane partitions \cite{vul}. 

Based on ideas of Kirillov and Noumi \cite{kn}, Warnaar \cite{war} observed that if one applies an appropriately specialized difference operator to both sides of the $m=n$ identity \eref{mac_cauchy}, one obtains the following:
\begin{align} 
\label{mac_cauchy_6V}
\sum_{\lambda} 
\prod_{i=0}^{\infty} \prod_{j=1}^{m_i(\lambda)} (1-t^j) 
P_{\lambda} (x_1,\dots,x_n;t) 
P_{\lambda} (y_1,\dots,y_n;t) 
\propto 
Z_{\rm ASM}(x_1,\dots,x_n; y_1,\dots,y_n;t).
\end{align}
This is the content of Theorem \ref{thm2}. On the right, apart from some simple proportionality factors, $Z_{\rm ASM}$ is the partition function for the six-vertex model on the square lattice with domain wall boundary conditions (the Izergin--Korepin determinant \cite{kor,ize}) -- alternatively, the partition function for ASMs with certain $x$ and $y$ weights ascribed to the entries. On the left, the infinite sum looks similar to the one appearing on the left of (\ref{mac_cauchy}) (the only difference is in \eref{mac_cauchy_6V} the product over $i$ starts at 0 -- we are counting parts of $\lambda$ of size 0 up to a maximal length of $\lambda$ equal to $n$). It can be seen as a generating function for appropriately weighted plane partitions refining in a sense those of Vuleti\'c. There is a Schur version of this identity (Theorem \ref{thm1}), though there is no limiting procedure one can take from the Hall--Littlewood to the Schur case. Both identities however are generalized by Warnaar's identity \cite{war} at the level of Macdonald polynomials, but we do not go into such generality in the present paper (the addition of another parameter, $q$, seems to complicate things unnecessarily for most of the results we discuss).  

Next there is the symplectic Cauchy identity (valid for $m \leq n$):
\begin{align}
\sum_{\lambda} 
s_{\lambda} (x_1,\dots,x_m)
sp_{\lambda} (y_1,\dots,y_n)  
= 
\frac{\prod_{1 \leq i<j \leq m} (1-x_i x_j)}
{\prod_{i=1}^{m} \prod_{j=1}^{n} (1-x_i y_j)(1- \frac{x_i}{y_j})}
\end{align} 
where $sp_{\lambda}$ is the irreducible character of the symplectic group $Sp(2n)$ indexed by partition $\lambda$ (see \cite{sun}). It too has an interpretation in terms of some class of plane partitions, and a refinement -- Theorem \ref{thm3} -- which is the expansion of the partition function for UASMs in terms of Schur polynomials and symplectic characters. We conjecture a $t$-generalization of it (Conjecture \ref{conj1}) which states that
\begin{multline} 
\sum_{\lambda} 
\prod_{i=0}^{\infty} 
\prod_{j=1}^{m_i(\lambda)} 
(1-t^j)
P_{\lambda}(x_1,\dots,x_n;t) 
K_{\lambda}(y_1,\dots,y_n;t) 
= 
Z_{\rm UASM} 
:= 
\\
\frac{\prod_{i,j=1}^{n} (1-t x_i y_j) (1-\frac{t x_i}{y_j}) }
{\prod_{1\leq i<j \leq n} (x_i-x_j) (y_i-y_j) (1 - t x_i x_j) (1 - \frac{1}{y_i y_j}) }
\det\left[ 
\frac{(1-t)}{(1- x_i y_j)(1-t x_i y_j)(1- \frac{x_i}{y_j})(1-\frac{t x_i}{y_j})} 
\right]_{1\leq i,j \leq n}
\end{multline}
where $P_{\lambda}$ is an ordinary (aforementioned) Hall--Littlewood polynomial, and $K_{\lambda}$ denotes the $BC_n$-symmetric Hall--Littlewood polynomial indexed by $\lambda$.

Finally, we discuss a Littlewood identity for Schur polynomials
\begin{align}
\sum_{\substack{\lambda\ \text{with} \\ \text{even columns}}} 
s_{\lambda}(x_1,\dots,x_n) 
= 
\prod_{1 \leq i < j \leq n} 
\frac{1}{1-x_i x_j}
\end{align}
and how its slight modification in the spirit of \eqref{mac_cauchy_6V} relates certain symmetric plane partitions with off-diagonally symmetric ASMs (OSASMs). This is Theorem \ref{thm4}. We conjecture a companion identity at the level of Hall--Littlewood polynomials, which is Conjecture \ref{conj2} (Pf stands for Pfaffian):
\begin{multline}
\sum_{\substack{\lambda\ \text{with} \\ \text{even columns}}} \ \ 
\prod_{i=0}^{\infty}\ \prod_{j=2,4,6,\dots}^{m_i(\lambda)} 
(1-t^{j-1}) 
P_{\lambda}(x_1,\dots,x_{2n};t) 
= 
Z_{\rm OSASM} 
:= 
\\
\prod_{1 \leq i<j \leq 2n} 
\frac{(1-t x_i x_j)}{(x_i-x_j)} 
{\rm Pf} \left[ 
\frac{(x_i-x_j)(1-t)}{(1-x_i x_j) (1-t x_i x_j)} 
\right]_{1\leq i < j \leq 2n}.
\end{multline}

The paper is organized as follows: in Sections \ref{sec:ASM} and \ref{sec:UASM} we discuss the Cauchy identities for Schur polynomials and symplectic characters and their modifications which lead to partition functions for ASMs and UASMs, respectively. In Section \ref{sec:ASM} we state Warnaar's Hall--Littlewood generalization of the result, and in Section \ref{sec:UASM} we conjecture a similar result involving $BC_n$-symmetric Hall--Littlewood polynomials. In Section \ref{sec:OSASM} we discuss the ``even columns'' Littlewood identity for Schur polynomials and how its appropriate modification relates to the partition function of OSASMs. We also conjecture a similar identity at the level of Hall--Littlewood polynomials. We conclude each of the three sections by remarking what happens when some of the variables (spectral parameters) are zero. This leads to six-vertex model configurations on rectangular domains (rectangular ASMs and UASMs) in Sections \ref{sec:ASM} and \ref{sec:UASM}, and to odd-sized OSASMs in Section \ref{sec:OSASM}.

Throughout the paper, $\b{x} := \frac{1}{x}$. We reserve letters $\lambda, \mu, \dots$ for partitions. A partition $\lambda$ is either the empty partition ($\emptyset$) or a sequence of strictly positive numbers listed in decreasing order: $\lambda_1 \geq \lambda_2 \geq \cdots \geq \lambda_k > 0$. We call each $\lambda_i$ a \textit{part} and $\ell(\lambda):=k$ the \textit{length} (number of non-zero parts) of $\lambda$. Moreover, we call $|\lambda|:=\sum_{i=1}^k \lambda_i$ the \textit{weight} of the partition. We finally define the notion of interlacing partitions. Let $\lambda$ and $\mu$ be two partitions with $|\lambda| \geq |\mu|$. They are said to be \textit{interlacing} and we write $\lambda \succ \mu$ if and only if 
$$ \lambda_1 \geq \mu_1 \geq \lambda_2 \geq \mu_2 \geq \lambda_3 \geq \dots$$


\section{Refined Cauchy identities and alternating sign matrices}
\label{sec:ASM}

\subsection{Cauchy determinant} Let $x_1,\dots,x_n$ and $y_1,\dots,y_n$ be indeterminates. We start with the following well known determinantal formula (the \textit{Cauchy determinant}):
\begin{align}
\det\left[ \frac{1}{1-x_i y_j} \right]_{1\leq i,j \leq n}
=
\frac{\Delta(x)_n \Delta(y)_n}{\prod_{i,j=1}^n (1-x_i y_j)}
\label{cauch-det}
\end{align}
where we have denoted $\Delta(x)_n:= \prod_{1 \leq i<j \leq n} (x_i - x_j)$ to be the Vandermonde product in the $x$ variables (and similarly for the $y$ variables). One proves this in the usual way by comparing zeros and poles of the two sides (a method called factor exhaustion).

\subsection{Weyl formula for Schur polynomials}

The Schur polynomials $s_{\lambda}(x_1,\dots,x_n)$ comprise one of the standard bases for the ring of symmetric functions in $n$ variables \cite{mac}. They are the characters of irreducible representations of $GL(n)$, and as such can be evaluated using the Weyl character formula. This gives rise to the determinant expression
\begin{align*}
s_{\lambda}(x_1,\dots,x_n)
=
\frac{
\det\left[ x_i^{\lambda_j-j+n} \right]_{1 \leq i,j \leq n}
}
{
\prod_{1 \leq i<j \leq n} (x_i - x_j)
}.
\end{align*}

\subsection{Tableau formula for Schur polynomials}

A {\it semi-standard Young tableau} (SSYT) of shape $\lambda$ is an assignment of one symbol $\{1,\dots,n\}$ to each box of the Young diagram $\lambda$, subject to the following rules:
\begin{enumerate}[{\bf 1.}]
\item The symbols have the ordering $1 < \cdots < n$.
\item The entries in $\lambda$ increase weakly along each row and strictly down each column.
\end{enumerate}
The Schur polynomial $s_{\lambda}(x_1,\dots,x_n)$ admits an alternative, combinatorial formula, as a weighted sum over semi-standard Young tableaux $T$ of shape $\lambda$:
\begin{align}
s_{\lambda}(x_1,\dots,x_n)
=
\sum_{T}
\prod_{k=1}^{n}
x_k^{\#(k)},
\label{schur-tab}
\end{align}
where $\#(k)$ denotes the number of occurrences of $k$ in $T$.

\subsection{Tableaux as interlacing sequences of partitions}

Any semi-standard Young tableau of shape $\lambda$ can be decomposed into a sequence of interlacing partitions:\footnote{Throughout the paper we distinguish different partitions by superscripts, \ie $\lambda^{(1)}, \lambda^{(2)}, \dots$, while the parts of a particular partition $\lambda$ are denoted using subscripts, \ie $\lambda_1, \lambda_2, \dots$, as is standard in the literature.} 
\begin{align*}
T
=
\{
\emptyset \equiv \lambda^{(0)} \prec \lambda^{(1)} \prec 
\cdots
\prec \lambda^{(n)} \equiv \lambda
\}.
\end{align*}
The correspondence is illustrated by the example in Figure \ref{fig-tab}.

\begin{figure}[H]
\begin{tabular}{lllllllll}
\begin{tikzpicture}[scale=0.5]
\draw (0,4) -- (5,4);
\draw (0,3) -- (5,3);
\draw (0,2) -- (3,2);
\draw (0,1) -- (3,1);
\draw (0,0) -- (1,0);
\draw (0,0) -- (0,4);
\draw (1,0) -- (1,4);
\draw (2,1) -- (2,4);
\draw (3,1) -- (3,4);
\draw (4,3) -- (4,4);
\draw (5,3) -- (5,4);
\node at (0.5,3.5) {\color{pink} 1};
\node at (1.5,3.5) {\color{pink} 1};
\node at (2.5,3.5) {\color{lred} 2};
\node at (3.5,3.5) {\color{lred} 2};
\node at (4.5,3.5) {\color{burg} 4};
\node at (0.5,2.5) {\color{lred} 2};
\node at (1.5,2.5) {\color{lred} 2};
\node at (2.5,2.5) {\color{dred} 3};
\node at (0.5,1.5) {\color{dred} 3};
\node at (1.5,1.5) {\color{dred} 3};
\node at (2.5,1.5) {\color{burg} 4};
\node at (0.5,0.5) {\color{burg} 4};
\end{tikzpicture}
&
&
{\color{pink}
\begin{tikzpicture}[scale=0.5]
\draw (0,4) -- (2,4);
\draw (0,3) -- (2,3);
\draw (0,3) -- (0,4);
\draw (1,3) -- (1,4);
\draw (2,3) -- (2,4);
\end{tikzpicture}
}
\quad\quad
&
&
{\color{lred}
\begin{tikzpicture}[scale=0.5]
\draw (0,4) -- (4,4);
\draw (0,3) -- (4,3);
\draw (0,2) -- (2,2);
\draw (0,2) -- (0,4);
\draw (1,2) -- (1,4);
\draw (2,2) -- (2,4);
\draw (3,3) -- (3,4);
\draw (4,3) -- (4,4);
\end{tikzpicture}
}
\quad\quad
&
&
{\color{dred}
\begin{tikzpicture}[scale=0.5]
\draw (0,4) -- (4,4);
\draw (0,3) -- (4,3);
\draw (0,2) -- (3,2);
\draw (0,1) -- (2,1);
\draw (0,1) -- (0,4);
\draw (1,1) -- (1,4);
\draw (2,1) -- (2,4);
\draw (3,2) -- (3,4);
\draw (4,3) -- (4,4);
\end{tikzpicture}
}
\quad\quad
&
&
{\color{burg}
\begin{tikzpicture}[scale=0.5]
\draw (0,4) -- (5,4);
\draw (0,3) -- (5,3);
\draw (0,2) -- (3,2);
\draw (0,1) -- (3,1);
\draw (0,0) -- (1,0);
\draw (0,0) -- (0,4);
\draw (1,0) -- (1,4);
\draw (2,1) -- (2,4);
\draw (3,1) -- (3,4);
\draw (4,3) -- (4,4);
\draw (5,3) -- (5,4);
\end{tikzpicture}
}
\\ \\
$T$ 
& = 
& {\color{pink} $\lambda^{(1)}$} 
& $\prec$
& {\color{lred} $\lambda^{(2)}$}
& $\prec$
& {\color{dred} $\lambda^{(3)}$}
& $\prec$
& {\color{burg} $\lambda^{(4)}$}
\end{tabular}
\caption{Decomposing a SSYT into interlacing partitions. By filling all boxes of the partition 
$\lambda^{(k)}$ with the symbol $k$, and stacking the partitions one on top of the other, we recover the tableau on the left hand side.}
\label{fig-tab}
\end{figure}
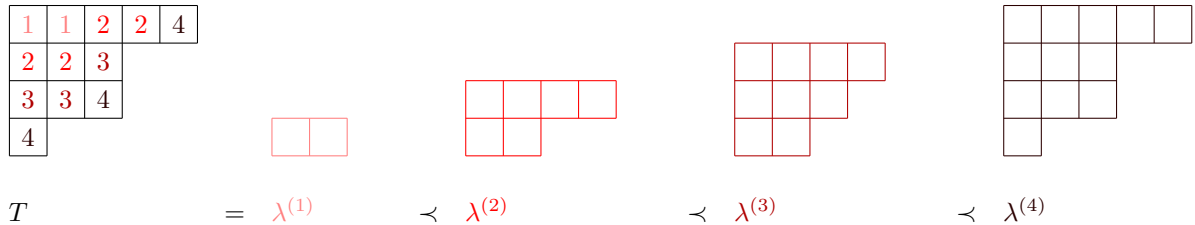

For any integer $k$, the number of instances of $k$ in the SSYT is given by the difference in the weight of successive partitions, $|\lambda^{(k)}| - |\lambda^{(k-1)}|$. Using this fact, the formula \eref{schur-tab} for the Schur polynomial becomes
\begin{align}
s_{\lambda}(x_1,\dots,x_n)
=
\sum_{T}
\prod_{k=1}^{n} x_k^{|\lambda^{(k)}|-|\lambda^{(k-1)}|}.
\label{schur-int}
\end{align}

\subsection{Plane partitions as conjoined tableaux}

A plane partition is a two dimensional array of non-negative integers $\pi(i,j)$, which satisfy the conditions 
\begin{align*}
\pi(i,j) \geq \pi(i+1,j),
\quad\quad
\pi(i,j) \geq \pi(i,j+1)
\end{align*}
for all $i,j \geq 1$, and such that $\pi(i,j)=0$ for $i,j$ sufficiently large.  By regarding each integer $\pi(i,j)$ as a column of cubes of height $\pi(i,j)$ over the Cartesian point $(i,j)$, one obtains a three dimensional visualization of a plane partition (see Figure \ref{fig:top_colors}).
\begin{figure}[H]
\includegraphics[scale=0.4]{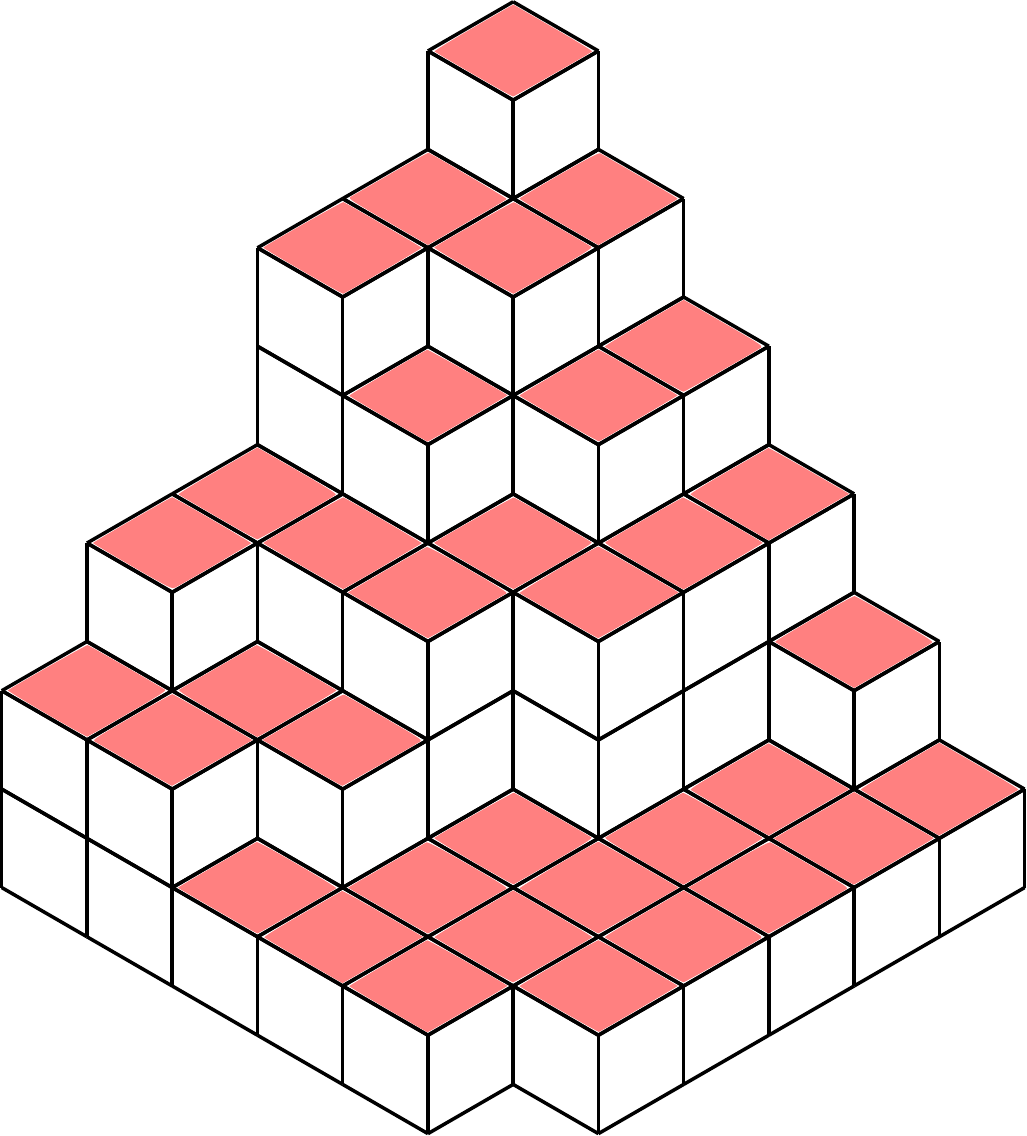} 
\caption{A three-dimensional representation of a plane partition. A column of $\pi(i,j)$ cubes is placed over the box $(i,j)$. All entries of size zero are not drawn.}
\label{fig:top_colors}
\end{figure}
An alternative formulation of plane partitions uses the notion of {\it interlacing diagonal slices} \cite{or}. Given a plane partition $\pi$ one can define its {\it diagonal slices} $\pi^{(i)}$, $i \in \mathbb{Z}$, which are partitions whose parts are given by
\begin{align*}
\pi^{(i)}_j
=
\left\{
\begin{array}{ll}
\pi(j-i,j),
&
i \leq 0
\\
\pi(j,i+j),
&
i \geq 0.
\end{array}
\right.
\end{align*}
The diagonal slices of any plane partition form an interlacing sequence of partitions
\begin{align*}
\cdots
\prec
\pi^{(-2)}
\prec
\pi^{(-1)}
\prec
\pi^{(0)}
\succ
\pi^{(1)}
\succ
\pi^{(2)}
\succ
\cdots
\end{align*}
where the slices become trivial, $\pi^{(i)} = \emptyset$, for sufficiently large $|i|$. Henceforth, we will always think of plane partitions from this point of view. We define the set $\bm{\pi}_{m,n}$ of plane partitions consisting of at most $m$ non-trivial left slices, and $n$ non-trivial right slices:
\begin{align}
\bm{\pi}_{m,n}
=
\{ 
\emptyset \equiv \lambda^{(0)} \prec \lambda^{(1)} \prec \cdots \prec \lambda^{(m)} 
\equiv 
\mu^{(n)} \succ \cdots \succ \mu^{(1)} \succ \mu^{(0)} \equiv \emptyset
\}.
\label{conj-tab}
\end{align}
Because of the identification of the partitions $\lambda^{(m)} = \mu^{(n)}$, one can think of \eref{conj-tab} as the set of all pairs of SSYT (in the alphabets $\{1,\dots,m\}$ and 
$\{1,\dots,n\}$, respectively) which have a common shape. In this sense, plane partitions are in one-to-one correspondence with {\it conjoined} SSYT.

\subsection{Cauchy identity for Schur polynomials and plane partitions}

Consider the Cauchy identity for Schur polynomials:
\begin{align}
\sum_{\lambda} s_{\lambda}(x_1,\dots,x_m) s_{\lambda}(y_1,\dots,y_n)
=
\prod_{i =1}^{m} \prod_{j=1}^{n} \frac{1}{1-x_i y_j}.
\label{s-cauch}
\end{align}
In view of the tableau rule \eref{schur-int} for Schur polynomials, the sum on the left hand side is taken over all pairs of conjoined SSYT. Since these correspond with plane partitions,
\eref{s-cauch} has a natural interpretation as a generating series of plane partitions:
\begin{align}
\label{s-pp-gs}
\sum_{\pi \in \bm{\pi}_{m,n} }
\prod_{i=1}^{m} x_i^{|\lambda^{(i)}|-|\lambda^{(i-1)}|}
\prod_{j=1}^{n} y_j^{|\mu^{(j)}|-|\mu^{(j-1)}|}
=
\prod_{i =1}^{m} \prod_{j=1}^{n} \frac{1}{1-x_i y_j}.
\end{align}
Taking the $q$-specialization $x_i = q^{m-i+1/2}$ and $y_j = q^{n-j+1/2}$, 
we recover volume-weighted plane partitions:
\begin{align}
\sum_{\pi \in \bm{\pi}_{m,n} }
q^{|\pi|}
=
\prod_{i=1}^{m}
\prod_{j=1}^{n}
\frac{1}{1-q^{m+n-i-j+1}}
=
\prod_{i=1}^{m}
\prod_{j=1}^{n}
\frac{1}{1-q^{i+j-1}}
\label{vol-pp}
\end{align}
Finally, taking the limit $m,n \rightarrow \infty$ of \eref{vol-pp} gives rise to the classical MacMahon generating series \cite{macm}:
\begin{align}
\sum_{\pi} q^{|\pi|}
=
\prod_{i=1}^{\infty}
\frac{1}{(1-q^i)^i}.
\label{macmahon}
\end{align}

\subsection{Hall--Littlewood polynomials}
\label{HL-def}

Hall--Littlewood polynomials are the common $t$-generalization of at least three families of symmetric polynomials. They can be defined as a sum over the symmetric group (see Chapter III of \cite{mac}):
\begin{align*}
P_{\lambda}(x_1,\dots,x_n;t)
=
\frac{1}{v_{\lambda}(t)}
\sum_{\sigma \in S_n}
\sigma \left(
x_1^{\lambda_1}
\dots
x_n^{\lambda_n}
\prod_{1\leq i<j \leq n}
\frac{x_i - t x_j}{x_i - x_j}
\right)
\end{align*}
where the function $v_{\lambda}(t)$ is given by
\begin{align*}
v_{\lambda}(t)
=
\prod_{i=0}^{\infty}
\prod_{j=1}^{m_i(\lambda)}
\frac{1-t^j}{1-t}.
\end{align*}
When $t=0$, one recovers the Schur polynomial $s_{\lambda}(x_1,\dots,x_n)$. On the other hand, setting $t=1$ gives the monomial symmetric polynomials 
$m_{\lambda}(x_1,\dots,x_n)$, while $t=-1$ gives rise to the Schur $P$-polynomials.

\subsection{Tableau rule for Hall--Littlewood polynomials}

Remarkably, Hall--Littlewood polynomials can also be expressed as a sum over SSYT \cite{mac}:
\begin{align}
\label{hl-tab}
P_{\lambda}(x_1,\dots,x_n;t)
=
\sum_{T}
\prod_{k=1}^{n}
\left(
x_k^{\#(k)}
\psi_{\lambda^{(k)} / \lambda^{(k-1)}}(t)
\right)
\end{align}
where the function $\psi_{\lambda/\mu}(t)$ is given by
\begin{align*}
\psi_{\lambda / \mu}(t)
=
\prod_{\substack{i \geq 1 \\ m_i(\mu) = m_i(\lambda)+1}}
\left(1-  t^{m_i(\mu)} \right)
\end{align*}
with the product taken over all $i\geq 1$, such that the multiplicity of the part $i$ in the partition $\mu$ is one greater than the multiplicity of that part in $\lambda$.

\subsection{Cauchy identity for Hall--Littlewood polynomials and $t$-weighted plane partitions}
\label{sec:HL-pp}

Consider the Cauchy identity for Hall--Littlewood polynomials,
\begin{align}
\sum_{\lambda}
b_{\lambda}(t)
P_{\lambda}(x_1,\dots,x_m;t)
P_{\lambda}(y_1,\dots,y_n;t)
=
\prod_{i=1}^{m}
\prod_{j=1}^{n}
\frac{1-t x_i y_j}{1-x_i y_j} ,
\label{hl-cauch1}
\end{align}
which can be written more explicitly as
\begin{align}
\sum_{\lambda}
\prod_{i=1}^{\infty}
\prod_{j=1}^{m_i(\lambda)}
(1-t^j)
P_{\lambda}(x_1,\dots,x_m;t)
P_{\lambda}(y_1,\dots,y_n;t)
=
\prod_{i=1}^{m}
\prod_{j=1}^{n}
\frac{1-t x_i y_j}{1-x_i y_j} .
\label{hl-cauch2}
\end{align}
Here too, one can study the Cauchy identity \eref{hl-cauch2} in the framework of plane partitions \cite{vul}. To do so, one must introduce the notion of {\it paths} on a plane partition.
\begin{defn}[Connected components, depth, paths]
Let $\pi$ be a plane partition. A connected component at height $h$ is a set of elements 
$\pi(i,j)=h$, whose coordinate boxes $(i,j)$ are connected. The depth\footnote{
In \cite{vul}, Vuleti\'c denoted this quantity {\it level} rather than {\it depth.} We have adopted the latter term, as we find it more intuitive.} 
of an element $\pi(i,j)>0$ is the smallest positive integer $d$ such that $\pi(i,j)$ and 
$\pi(i+d,j+d)$ are not in the same connected component.\footnote{
Obviously this does not lead to a well-defined value for the depth of elements $\pi(i,j)=0$, but this does not concern us for the moment. We extend the definition to height-0 elements in Section \ref{sec:refined-pp}.} 
A path at height $h$ and depth $d$ is a connected subset of a connected component at height $h$, whose elements have the same depth, $d$. We let $p_{d}(\pi)$ denote the number of paths in $\pi$ with height $h>0$ and depth $d$. 
\end{defn}

\begin{figure}[H]
\begin{tabular}{cccc}
\includegraphics[scale=0.4]{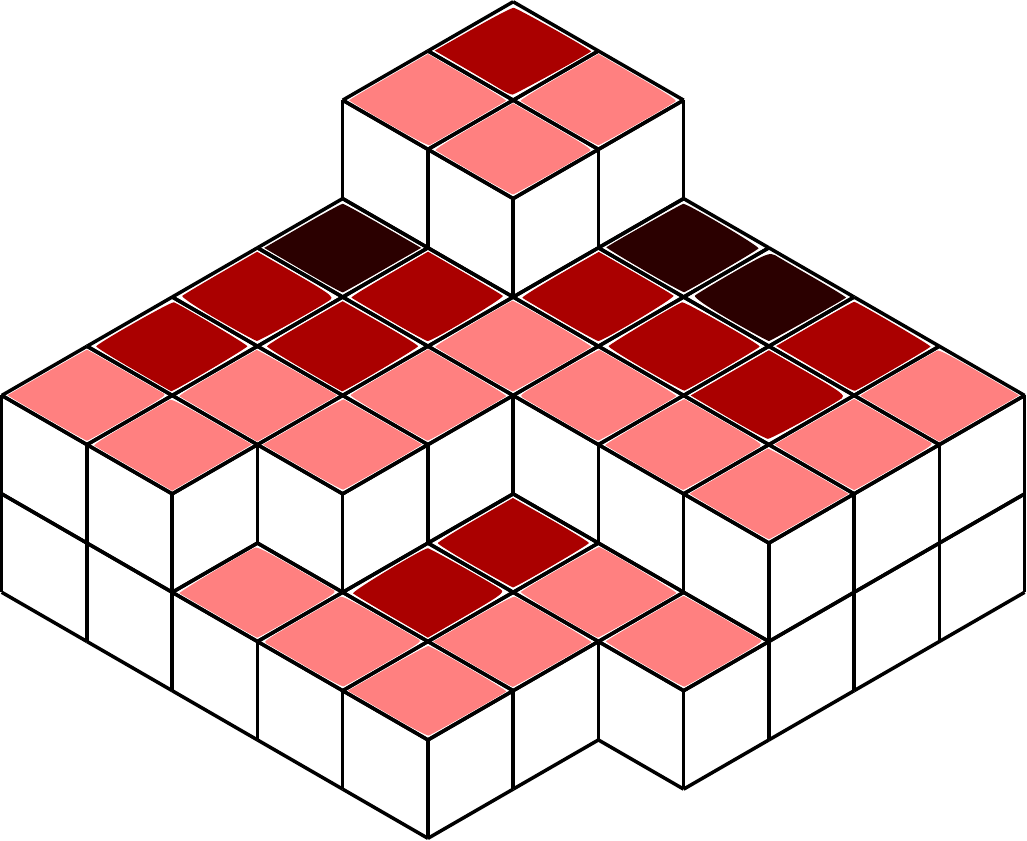}
& & &
\begin{tikzpicture}[scale=0.4]
\draw[fill=pink] (5,0) rectangle (6,1);
\draw[fill=dred] (5,2) rectangle (6,3);
\draw[fill=burg] (5,4) rectangle (6,5);
\node at (8.5,0.5) {Depth-1};
\node at (8.5,2.5) {Depth-2};
\node at (8.5,4.5) {Depth-3};
\end{tikzpicture}
\end{tabular}
\caption{A plane partition with 3 connected components. Each connected component can be subdivided into paths, at varying depths. A path of depth $k$ receives a weighting of 
$1-t^k$. This plane partition has 3 depth-1 paths, 4 depth-2 paths and 2 depth-3 paths, so it receives the $t$-weighting $(1-t)^3 (1-t^2)^4 (1-t^3)^2$.}
\label{fig:HL_levels}
\end{figure}
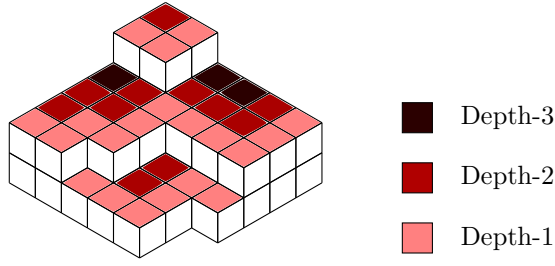

Using the tableau formula \eref{hl-tab} for Hall--Littlewood polynomials, one can interpret the left hand side of \eref{hl-cauch2} as a generating series of path-weighted plane partitions (see Figure \ref{fig:HL_levels}):
\begin{align}
\label{hl-pp-gs}
\sum_{\pi \in \bm{\pi}_{m,n}}
\prod_{i \geq 1} \left(1-t^i\right)^{p_i(\pi)}
\prod_{i=1}^{m} x_i^{|\lambda^{(i)}|-|\lambda^{(i-1)}|}
\prod_{j=1}^{n} y_j^{|\mu^{(j)}|-|\mu^{(j-1)}|}
=
\prod_{i=1}^{m}
\prod_{j=1}^{n}
\frac{1-t x_i y_j}{1-x_i y_j}.
\end{align}
Taking the $q$-specialization $x_i = q^{m-i+1/2}$ and $y_j = q^{n-j+1/2}$ as previously, one obtains a 
$t$-refinement of the generating series \eref{vol-pp}:
\begin{align*}
\sum_{\pi \in \bm{\pi}_{m,n}}
\prod_{i \geq 1} \left(1-t^i\right)^{p_i(\pi)}
q^{|\pi|}
=
\prod_{i=1}^{m}
\prod_{j=1}^{n}
\frac{1-t q^{i+j-1}}{1-q^{i+j-1}}.
\end{align*}
Furthermore, passing to the limit $m,n \rightarrow \infty$, we obtain a $t$-refinement of the MacMahon formula \eref{macmahon}:
\begin{align}
\label{vuletic-gs}
\sum_{\pi}
\prod_{i \geq 1} \left(1-t^i\right)^{p_i(\pi)}
q^{|\pi|}
=
\prod_{i=1}^{\infty}
\left(
\frac{1-t q^i}{1-q^i}
\right)^i.
\end{align}
The generating series \eref{vuletic-gs} was originally obtained in \cite{vul}. For a proof of this result using the formalism of vertex operators, see \cite{fw}.

\subsection{Macdonald's difference operators}

Let us briefly recall the construction of Macdonald's difference operators, following Chapter VI, Section 3 of \cite{mac}. We begin by defining the operator $T_{u,x_i}$, which is an endomorphism on the space of functions in $n$ variables $(x_1,\dots,x_n)$, whose action is to scale the variable $x_i$ by the parameter $u$:
\begin{align*}
T_{u,x_i} : f(x_1,\dots,x_n) \mapsto f(x_1,\dots,ux_i,\dots,x_n) . 
\end{align*}
From this, one can define a family of difference operators
\begin{align}
D_n^r
=
\sum_{
\substack{
S \subseteq [n] \\ 
|S| = r
}
}
t^{r(r-1)/2}
\prod_{\substack{
i \in S \\ j \not\in S
}}
\frac{tx_i - x_j}{x_i - x_j}
\prod_{i \in S}
T_{q,x_i}
\label{do}
\end{align}
where the sum is taken over all subsets $S \subseteq \{1,\dots,n\}$ of cardinality $r$. These operators can in turn be grouped in a generating series with free parameter $z$:
\begin{align}
D_n(z;q,t)
=
\sum_{r=0}^{n}
D_n^{r} z^r .
\label{do-gs}
\end{align}
The Macdonald polynomials $P_{\lambda}(x;q,t)$ are eigenfunctions of the difference operators:
\begin{align}
D_n(z;q,t) P_{\lambda}(x;q,t)
=
\prod_{i=1}^{n} (1+z q^{\lambda_i} t^{n-i})
P_{\lambda}(x;q,t).
\label{mac-eig}
\end{align}
In what follows we will be interested in the special cases $q=t$ and $q=0$ of \eref{mac-eig}, which correspond to the restriction to Schur and Hall--Littlewood polynomials, respectively.

\subsection{Refined Cauchy identities}
\label{sec:refined-cauchy}

We now state and prove two theorems central to the paper. Both are refinements of the Cauchy identity 
\eref{s-cauch} (in the strict sense that both degenerate to it when $t=0$, which can be easily deduced using the Cauchy formula \eref{cauch-det}). The first theorem involves Schur polynomials while the second Hall--Littlewood polynomials. The second theorem can also be thought of as a deformation of the Cauchy identity for Hall--Littlewood polynomials \eref{hl-cauch2} (in view of the strong similarities between the two), although it is not a refinement of \eref{hl-cauch2} in the strict sense (since it does not degenerate to \eref{hl-cauch2} in any limit).  

The first theorem is the expansion of the Izergin--Korepin determinant \cite{kor,ize} in terms of Schur polynomials. It is well known to experts in the area and follows trivially from results of Warnaar \cite{war}. 

\begin{thm}
\label{thm1}

\begin{align}
\label{s-cauchy-refine}
\sum_{\lambda} 
\prod_{i=1}^{n} (1-t^{\lambda_i-i+n+1})
s_{\lambda}(x_1,\dots,x_n) s_{\lambda}(y_1,\dots,y_n)
=
\frac{1}{\Delta(x)_n \Delta(y)_n}
\det\left[
\frac{(1-t)}{(1-x_i y_j)(1-t x_i y_j)}
\right]_{1\leq i,j \leq n}.
\end{align}

\end{thm}

\begin{proof} 
This can be proved in at least two ways. One possibility is to act on the $m=n$ Cauchy identity \eref{s-cauch} with the generating series \eref{do-gs}, evaluated at $z=-t$ and $q=t$. We give a simpler proof, by treating the entries of the determinant as formal power series:
\begin{align*}
\det\left[
\frac{(1-t)}{(1-x_i y_j)(1-t x_i y_j)}
\right]_{1\leq i,j \leq n}
&=
\det\left[
\sum_{k=0}^{\infty}
(1-t^{k+1}) x_i^k y_j^k
\right]_{1\leq i,j \leq n}
\\
&=
\sum_{k_1 > \cdots > k_n \geq 0}
\det\left[
x_i^{k_j}
\right]_{1\leq i,j \leq n}
\det\left[
y_j^{k_i}
\right]_{1\leq i,j \leq n}
\prod_{i=1}^{n}
(1-t^{k_i+1})
\end{align*}
where the final line follows from the Cauchy--Binet identity. Shifting the summation indices in the standard way $k_i = \lambda_i - i + n$, we obtain
\begin{align*}
\det\left[
\frac{(1-t)}{(1-x_i y_j)(1-t x_i y_j)}
\right]_{1\leq i,j \leq n}
=
\sum_{\lambda}
\det\left[
x_i^{\lambda_j-j+n}
\right]_{1\leq i,j \leq n}
\det\left[
y_j^{\lambda_i-i+n}
\right]_{1\leq i,j \leq n}
\prod_{i=1}^{n}
(1-t^{\lambda_i-i+n+1})
\end{align*}
and the proof is complete after dividing by the Vandermonde factor 
$\Delta(x)_n \Delta(y)_n$. 

\end{proof}

We now state the Hall--Littlewood case (not the most general, as in \cite{war} this is done at the level of Macdonald polynomials): 

\begin{thm}[Kirillov--Noumi, Warnaar]
\label{thm2}

\begin{multline}
\sum_{\lambda}
\prod_{i=0}^{\infty}
\prod_{j=1}^{m_i(\lambda)}
(1-t^j)
P_{\lambda}(x_1,\dots,x_n;t)
P_{\lambda}(y_1,\dots,y_n;t)
\\
=
\frac{\prod_{i,j=1}^{n} (1-t x_i y_j)}
{\Delta(x)_n \Delta(y)_n}
\det
\left[
\frac{(1-t)}{(1-x_i y_j)(1-t x_i y_j)}
\right]_{1\leq i,j \leq n}.
\label{knw-id}
\end{multline} 

\end{thm}

\begin{proof}
We give the same proof as \cite{kn,war}. Acting on the $m=n$ Cauchy identity \eref{hl-cauch1} with the generating series \eref{do-gs}, evaluated at $z=-t$ and $q=0$, one obtains
\begin{align*}
D_n(-t;0,t)
\prod_{i,j=1}^n 
\frac{1-tx_i y_j}{1-x_i y_j}
&=
\sum_{\lambda}
b_{\lambda}(t)
\prod_{i=1}^{n}
(1-\delta_{\lambda_i,0} t^{n-i+1})
P_{\lambda}(x_1,\dots,x_n;t)
P_{\lambda}(y_1,\dots,y_n;t)
\\
&=
\sum_{\lambda}
\prod_{i=0}^{\infty}
\prod_{j=1}^{m_i(\lambda)}
(1-t^j)
P_{\lambda}(x_1,\dots,x_n;t)
P_{\lambda}(y_1,\dots,y_n;t).
\end{align*}
It remains to show that the difference operators act on the Hall--Littlewood Cauchy kernel to produce the determinant on the right hand side of \eref{knw-id}. To achieve that, we start from the determinant itself:
\begin{align*}
\det\left[
\frac{(1-t)}{(1-x_i y_j)(1-t x_i y_j)}
\right]_{1\leq i,j \leq n}
&=
\det\left[
\frac{1}{1-x_i y_j}
-
\frac{t}{1-t x_i y_j}
\right]_{1\leq i,j \leq n}
\\
&
=
\sum_{r=0}^{n} 
\sum_{\substack{
S \subseteq [n] \\ |S| = r
}}
(-t)^r 
\det\left[
\frac{1}{1-x_i^S y_j} 
\right]_{1\leq i,j \leq n}
\nonumber
\end{align*}
where $x_i^S = t x_i$ if $i \in S$, and $x_i^S = x_i$ if $i \not\in S$. Replacing each determinant in this sum with its factorized form \eref{cauch-det}, we find that
\begin{align*}
\det\left[
\frac{(1-t)}{(1-x_i y_j)(1-t x_i y_j)}
\right]_{1\leq i,j \leq n}
=
\sum_{r=0}^n
\sum_{\substack{S \subseteq [n] \\ |S| = r}}
(-t)^r
\frac{\prod_{1 \leq i < j \leq n} (x_i^S - x_j^S) (y_i - y_j)}
{\prod_{i,j=1}^n (1-x_i^S y_j)}
\end{align*}
or more explicitly, using the definition of the symbol $x_i^S$,
\begin{align*}
\det\left[
\frac{(1-t)}{(1-x_i y_j)(1-t x_i y_j)}
\right]_{1\leq i,j \leq n}
=
\prod_{1 \leq i < j \leq n} (y_i - y_j)
\sum_{r=0}^n
\sum_{\substack{S \subseteq [n] \\ |S| = r}}
(-1)^{\mathcal{I}(S)} (-1)^r t^{r(r+1)/2}
\times
\\
\prod_{\substack{i<j \\ i,j \in S}}
(x_i - x_j)
\prod_{\substack{i<j \\ i,j \not\in S}}
(x_i - x_j)
\prod_{\substack{i \in S \\ j \not\in S}}
(t x_i - x_j)
\prod_{j=1}^n
\prod_{i \in S}
\frac{1}{1-tx_i y_j}
\prod_{j=1}^n
\prod_{i \not\in S}
\frac{1}{1-x_i y_j}
\nonumber
\end{align*}
where $\mathcal{I}(S) = \{\# (i,j) | i > j, i \in S, j \not\in S \}$. Restoring the prefactors, we obtain
\begin{multline*}
\frac{\prod_{i,j=1}^n (1-tx_i y_j)}{\Delta(x)_n \Delta(y)_n}
\det\left[
\frac{(1-t)}{(1-x_i y_j)(1-t x_i y_j)}
\right]_{1\leq i,j \leq n}
=
\\
\left(
\sum_{r=0}^n
\sum_{\substack{S \subseteq [n] \\ |S| = r}}
(-1)^r t^{r(r+1)/2}
\prod_{\substack{i \in S \\ j \not\in S}}
\frac{t x_i - x_j}{x_i - x_j}
\prod_{j=1}^n
\prod_{i \in S}
\frac{1-x_i y_j}{1-tx_i y_j}
\right)
\prod_{i,j=1}^n 
\frac{1-tx_i y_j}{1-x_i y_j}.
\end{multline*}
Finally, referring to the definition \eref{do}, \eref{do-gs} of $D_n(-t;0,t)$, it is clear that
\begin{align*}
\frac{\prod_{i,j=1}^n (1-tx_i y_j)}{\Delta(x)_n \Delta(y)_n}
\det\left[
\frac{(1-t)}{(1-x_i y_j)(1-t x_i y_j)}
\right]_{1\leq i,j \leq n}
=
D_n(-t;0,t)
\prod_{i,j=1}^n 
\frac{1-tx_i y_j}{1-x_i y_j}.
\end{align*}

\end{proof}
 
\subsection{Refined plane partitions}
\label{sec:refined-pp}

We now study the combinatorial meaning of equations \eref{s-cauchy-refine} and 
\eref{knw-id}. The first of these, \eref{s-cauchy-refine}, can be interpreted as a 
$t$-refinement of the generating series \eref{s-pp-gs} for plane partitions. This refinement assigns an extra $t$-dependent weight to the central slice of a plane partition:
\begin{multline*}
\sum_{\pi \in \bm{\pi}_{n,n}}
\prod_{i=1}^{n} (1-t^{\pi(i,i)-i+n+1})
\prod_{i=1}^{n} 
x_i^{|\lambda^{(i)}|-|\lambda^{(i-1)}|}
y_i^{|\mu^{(i)}|-|\mu^{(i-1)}|}
=
\\
\frac{1}{\Delta(x)_n \Delta(y)_n}
\det\left[
\frac{(1-t)}{(1-x_i y_j)(1-t x_i y_j)}
\right]_{1\leq i,j \leq n}.
\end{multline*}
It is a refinement in the true sense, since by setting $t=0$ one recovers the standard result \eref{s-pp-gs} at $m=n$. 

Turning to the second equation, \eref{knw-id}, we find that it has an even nicer combinatorial interpretation as a modification of the generating series \eref{hl-pp-gs}. To explain this modification, we need the following definition.

\begin{defn}[Paths at height zero]
\label{defn:height0}
Let $\pi$ be a plane partition whose base is contained within an $n \times n$ square, meaning that $\pi(i,j) > 0$ only for $i,j \leq n$. A path at height 0 and depth $d$ is the set of all elements $\pi(i,j)=0$ such that ${\rm max}(i,j)=n-d+1$. We let 
$\tilde{p}_d(\pi)_{n \times n}$ denote the number of paths in $\pi$ with height $h \geq 0$ and depth $d$.\footnote{
Notice that the depth of the paths at height 0 depends on the framing of the plane partition. For this reason $\tilde{p}_d(\pi)_{n \times n}$ depends explicitly on $n$.}
\end{defn}

We now assign a $t$-weighting to plane partitions in much the same way as before, with each path of depth $k$ receiving a weight of $1-t^k$. The sole difference is that height-0 entries are now treated on much the same footing as the rest. This is more clearly illustrated in Figures \ref{fig:t-refine} and \ref{fig:t-refine-3d}.

\begin{figure}
\begin{tabular}{ccc}
\begin{tikzpicture}[scale=0.5,rotate=45]
\draw[] (0,0) rectangle (1,1);
\draw[] (0,1) rectangle (1,2);
\draw[] (0,2) rectangle (1,3);
\draw[] (0,3) rectangle (1,4);
\draw[] (0,4) rectangle (1,5);
\draw[] (0,5) rectangle (1,6);
\draw[fill=pink] (0,6) rectangle (1,7);
\draw[] (1,0) rectangle (2,1);
\draw[] (1,1) rectangle (2,2);
\draw[] (1,2) rectangle (2,3);
\draw[] (1,3) rectangle (2,4);
\draw[fill=pink] (1,4) rectangle (2,5);
\draw[fill=pink] (1,5) rectangle (2,6);
\draw[fill=pink] (1,6) rectangle (2,7);
\draw[] (2,0) rectangle (3,1);
\draw[] (2,1) rectangle (3,2);
\draw[] (2,2) rectangle (3,3);
\draw[fill=pink] (2,3) rectangle (3,4);
\draw[fill=pink] (2,4) rectangle (3,5);
\draw[fill=pink] (2,5) rectangle (3,6);
\draw[fill=pink] (2,6) rectangle (3,7);
\draw[] (3,0) rectangle (4,1);
\draw[fill=pink] (3,1) rectangle (4,2);
\draw[fill=pink] (3,2) rectangle (4,3);
\draw[fill=pink] (3,3) rectangle (4,4);
\draw[fill=pink] (3,4) rectangle (4,5);
\draw[fill=pink] (3,5) rectangle (4,6);
\draw[fill=pink] (3,6) rectangle (4,7);
\draw[] (4,0) rectangle (5,1);
\draw[fill=pink] (4,1) rectangle (5,2);
\draw[fill=dred] (4,2) rectangle (5,3);
\draw[fill=dred] (4,3) rectangle (5,4);
\draw[fill=pink] (4,4) rectangle (5,5);
\draw[fill=pink] (4,5) rectangle (5,6);
\draw[fill=pink] (4,6) rectangle (5,7);
\draw[fill=pink] (5,0) rectangle (6,1);
\draw[fill=pink] (5,1) rectangle (6,2);
\draw[fill=dred] (5,2) rectangle (6,3);
\draw[fill=pink] (5,3) rectangle (6,4);
\draw[fill=pink] (5,4) rectangle (6,5);
\draw[fill=pink] (5,5) rectangle (6,6);
\draw[fill=pink] (5,6) rectangle (6,7);
\draw[fill=pink] (6,0) rectangle (7,1);
\draw[fill=pink] (6,1) rectangle (7,2);
\draw[fill=pink] (6,2) rectangle (7,3);
\draw[fill=pink] (6,3) rectangle (7,4);
\draw[fill=pink] (6,4) rectangle (7,5);
\draw[fill=dred] (6,5) rectangle (7,6);
\draw[fill=dred] (6,6) rectangle (7,7);


\node at (0.5,6.5) {2};

\node at (1.5,6.5) {2};
\node at (0.5,5.5) {0};

\node at (2.5,6.5) {4};
\node at (1.5,5.5) {2};
\node at (0.5,4.5) {0};

\node at (3.5,6.5) {5};
\node at (2.5,5.5) {3};
\node at (1.5,4.5) {1};
\node at (0.5,3.5) {0};

\node at (4.5,6.5) {5};
\node at (3.5,5.5) {3};
\node at (2.5,4.5) {1};
\node at (1.5,3.5) {0};
\node at (0.5,2.5) {0};

\node at (5.5,6.5) {5};
\node at (4.5,5.5) {4};
\node at (3.5,4.5) {2};
\node at (2.5,3.5) {1};
\node at (1.5,2.5) {0};
\node at (0.5,1.5) {0};

\node at (6.5,6.5) {\color{white} 5};
\node at (5.5,5.5) {5};
\node at (4.5,4.5) {3};
\node at (3.5,3.5) {1};
\node at (2.5,2.5) {0};
\node at (1.5,1.5) {0};
\node at (0.5,0.5) {0};

\node at (6.5,5.5) {\color{white} 5};
\node at (5.5,4.5) {5};
\node at (4.5,3.5) {\color{white} 1};
\node at (3.5,2.5) {1};
\node at (2.5,1.5) {0};
\node at (1.5,0.5) {0};

\node at (6.5,4.5) {5};
\node at (5.5,3.5) {2};
\node at (4.5,2.5) {\color{white} 1};
\node at (3.5,1.5) {1};
\node at (2.5,0.5) {0};

\node at (6.5,3.5) {3};
\node at (5.5,2.5) {\color{white} 1};
\node at (4.5,1.5) {1};
\node at (3.5,0.5) {0};

\node at (6.5,2.5) {2};
\node at (5.5,1.5) {1};
\node at (4.5,0.5) {0};

\node at (6.5,1.5) {2};
\node at (5.5,0.5) {1};

\node at (6.5,0.5) {1};
\end{tikzpicture}
& &
\begin{tikzpicture}[scale=0.5,rotate=45]
\draw[fill=pink] (0,0) rectangle (1,1);
\draw[fill=pink] (0,1) rectangle (1,2);
\draw[fill=pink] (0,2) rectangle (1,3);
\draw[fill=pink] (0,3) rectangle (1,4);
\draw[fill=pink] (0,4) rectangle (1,5);
\draw[fill=pink] (0,5) rectangle (1,6);
\draw[fill=pink] (0,6) rectangle (1,7);
\draw[fill=pink] (1,0) rectangle (2,1);
\draw[fill=dred] (1,1) rectangle (2,2);
\draw[fill=dred] (1,2) rectangle (2,3);
\draw[fill=dred] (1,3) rectangle (2,4);
\draw[fill=pink] (1,4) rectangle (2,5);
\draw[fill=pink] (1,5) rectangle (2,6);
\draw[fill=pink] (1,6) rectangle (2,7);
\draw[fill=pink] (2,0) rectangle (3,1);
\draw[fill=dred] (2,1) rectangle (3,2);
\draw[fill=burg] (2,2) rectangle (3,3);
\draw[fill=pink] (2,3) rectangle (3,4);
\draw[fill=pink] (2,4) rectangle (3,5);
\draw[fill=pink] (2,5) rectangle (3,6);
\draw[fill=pink] (2,6) rectangle (3,7);
\draw[fill=pink] (3,0) rectangle (4,1);
\draw[fill=pink] (3,1) rectangle (4,2);
\draw[fill=pink] (3,2) rectangle (4,3);
\draw[fill=pink] (3,3) rectangle (4,4);
\draw[fill=pink] (3,4) rectangle (4,5);
\draw[fill=pink] (3,5) rectangle (4,6);
\draw[fill=pink] (3,6) rectangle (4,7);
\draw[fill=pink] (4,0) rectangle (5,1);
\draw[fill=pink] (4,1) rectangle (5,2);
\draw[fill=dred] (4,2) rectangle (5,3);
\draw[fill=dred] (4,3) rectangle (5,4);
\draw[fill=pink] (4,4) rectangle (5,5);
\draw[fill=pink] (4,5) rectangle (5,6);
\draw[fill=pink] (4,6) rectangle (5,7);
\draw[fill=pink] (5,0) rectangle (6,1);
\draw[fill=pink] (5,1) rectangle (6,2);
\draw[fill=dred] (5,2) rectangle (6,3);
\draw[fill=pink] (5,3) rectangle (6,4);
\draw[fill=pink] (5,4) rectangle (6,5);
\draw[fill=pink] (5,5) rectangle (6,6);
\draw[fill=pink] (5,6) rectangle (6,7);
\draw[fill=pink] (6,0) rectangle (7,1);
\draw[fill=pink] (6,1) rectangle (7,2);
\draw[fill=pink] (6,2) rectangle (7,3);
\draw[fill=pink] (6,3) rectangle (7,4);
\draw[fill=pink] (6,4) rectangle (7,5);
\draw[fill=dred] (6,5) rectangle (7,6);
\draw[fill=dred] (6,6) rectangle (7,7);


\node at (0.5,6.5) {2};

\node at (1.5,6.5) {2};
\node at (0.5,5.5) {0};

\node at (2.5,6.5) {4};
\node at (1.5,5.5) {2};
\node at (0.5,4.5) {0};

\node at (3.5,6.5) {5};
\node at (2.5,5.5) {3};
\node at (1.5,4.5) {1};
\node at (0.5,3.5) {0};

\node at (4.5,6.5) {5};
\node at (3.5,5.5) {3};
\node at (2.5,4.5) {1};
\node at (1.5,3.5) {\color{white} 0};
\node at (0.5,2.5) {0};

\node at (5.5,6.5) {5};
\node at (4.5,5.5) {4};
\node at (3.5,4.5) {2};
\node at (2.5,3.5) {1};
\node at (1.5,2.5) {\color{white} 0};
\node at (0.5,1.5) {0};

\node at (6.5,6.5) {\color{white} 5};
\node at (5.5,5.5) {5};
\node at (4.5,4.5) {3};
\node at (3.5,3.5) {1};
\node at (2.5,2.5) {\color{white} 0};
\node at (1.5,1.5) {\color{white} 0};
\node at (0.5,0.5) {0};

\node at (6.5,5.5) {\color{white} 5};
\node at (5.5,4.5) {5};
\node at (4.5,3.5) {\color{white} 1};
\node at (3.5,2.5) {1};
\node at (2.5,1.5) {\color{white} 0};
\node at (1.5,0.5) {0};

\node at (6.5,4.5) {5};
\node at (5.5,3.5) {2};
\node at (4.5,2.5) {\color{white} 1};
\node at (3.5,1.5) {1};
\node at (2.5,0.5) {0};

\node at (6.5,3.5) {3};
\node at (5.5,2.5) {\color{white} 1};
\node at (4.5,1.5) {1};
\node at (3.5,0.5) {0};

\node at (6.5,2.5) {2};
\node at (5.5,1.5) {1};
\node at (4.5,0.5) {0};

\node at (6.5,1.5) {2};
\node at (5.5,0.5) {1};

\node at (6.5,0.5) {1};
\end{tikzpicture}
\end{tabular}
\caption{On the left, a plane partition which receives a $t$-weighting based on the paths within the non-zero entries. On the right, the same plane partition, which receives an additional $t$-weighting based on the paths among the zero entries.}
\label{fig:t-refine}
\end{figure}

\begin{figure}
\begin{tabular}{ccc}
\includegraphics[scale=0.4]{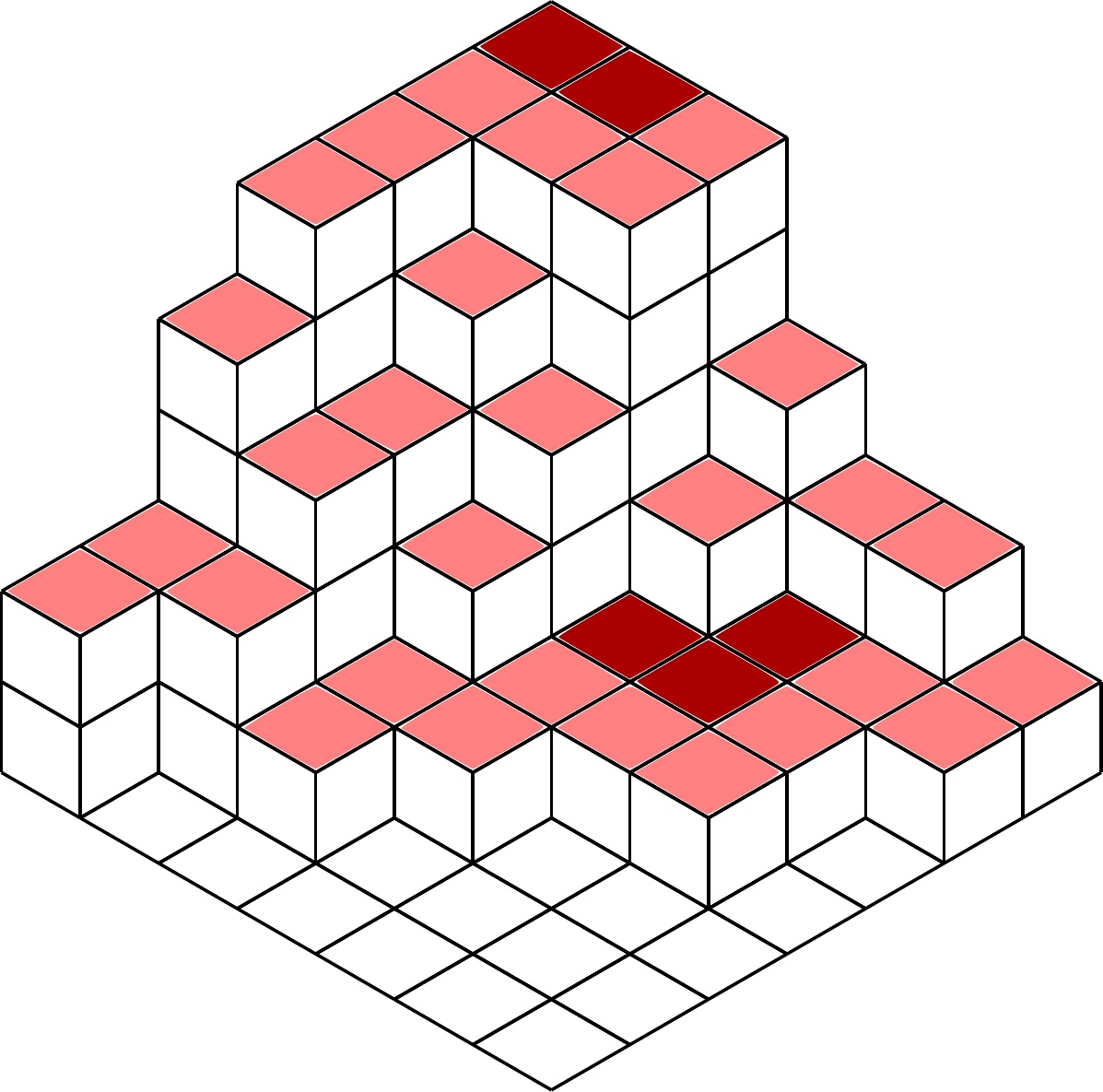}
& &
\includegraphics[scale=0.4]{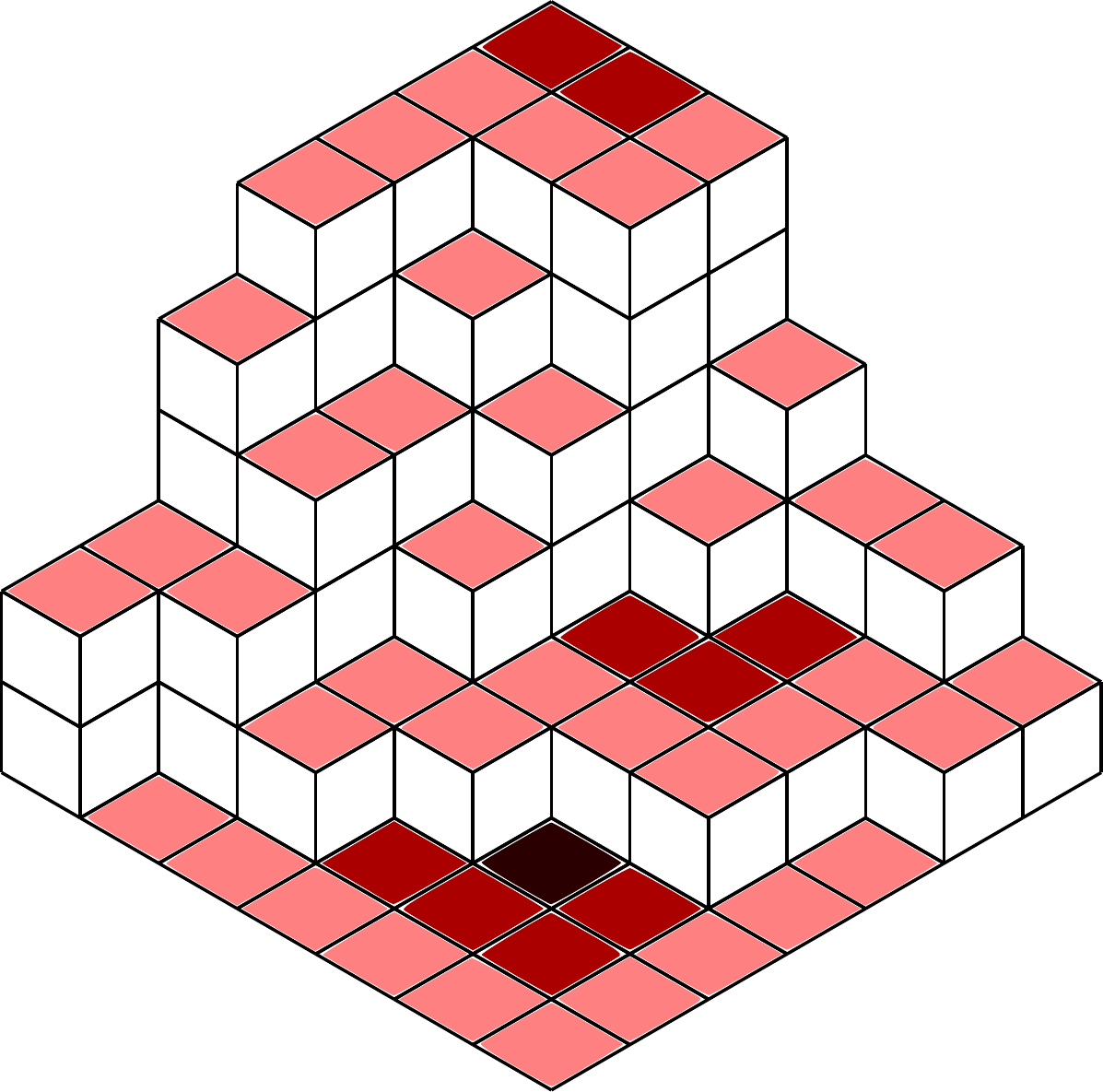}
\end{tabular}
\caption{The three-dimensional version of Figure \ref{fig:t-refine}. On the left, a path-weighted plane partition. On the right, the additional weighting which it receives due to the modified Cauchy identity \eref{knw-id}.}
\label{fig:t-refine-3d}
\end{figure}

One can easily show that the left hand side of \eref{knw-id} can be interpreted as a generating series of plane partitions with precisely the additional $t$-weighting described above. We find that
\begin{multline}
\label{pp-ASM-gs}
\sum_{\pi \in \bm{\pi}_{n,n}}
\prod_{i \geq 1} \left(1-t^i\right)^{\tilde{p}_i(\pi)_{n\times n}}
\prod_{i=1}^{n} x_i^{|\lambda^{(i)}|-|\lambda^{(i-1)}|} y_i^{|\mu^{(i)}|-|\mu^{(i-1)}|}
=
\\
\frac{\prod_{i,j=1}^{n} (1-t x_i y_j)}
{\Delta(x)_n \Delta(y)_n}
\det
\left[
\frac{(1-t)}{(1-x_i y_j)(1-t x_i y_j)}
\right]_{1\leq i,j \leq n}.
\end{multline}
Notice that this modification is not, strictly speaking, a refinement of the original generating series \eref{hl-pp-gs}, since it cannot be recovered as a special case.

\subsection{Six-vertex model with domain wall boundary conditions}
\label{sec:6v}

We now turn to the six-vertex model of statistical mechanics.\footnote{It is not our intention to describe this model in great detail, but rather to present the bare facts which are relevant to this work. For more information about the model and its solution, we refer the reader to \cite{bax}.} This is a model on the square lattice, with vertices formed by the intersection of lines. To each of the four edges surrounding an intersection, one assigns arrows which can point either towards the vertex center or away from it. The arrow configurations are constrained by the fact that each vertex must have two arrows which point towards its center, and two which point away. This gives rise to six possible vertices, as shown in Figure \ref{6vertices}. 
\begin{figure}[H]
\begin{tabular}{ccc}
\begin{tikzpicture}[scale=0.6,>=stealth]
\draw[thick, smooth] (-1,0) -- (1,0);
\node at (-0.5,0) {$\r$}; \node at (0.5,0) {$\r$};
\node[label={left: \fs ${\color{red} \shortrightarrow} \ x$}] at (-1,0) {};
\draw[thick, smooth] (0,-1) -- (0,1);
\node at (0,-0.5) {$\u$}; \node at (0,0.5) {$\u$};
\node[label={below: \fs ${\color{red} \begin{array}{c} {\color{black} y} \\ \shortuparrow \end{array} }$}] at (0,-1) {};
\end{tikzpicture}
\quad\quad\quad&
\begin{tikzpicture}[scale=0.6,>=stealth]
\draw[thick, smooth] (-1,0) -- (1,0);
\node at (-0.5,0) {$\r$}; \node at (0.5,0) {$\r$};
\node[label={left: \fs ${\color{red} \shortrightarrow} \ x$}] at (-1,0) {};
\draw[thick, smooth] (0,-1) -- (0,1);
\node at (0,-0.5) {$\d$}; \node at (0,0.5) {$\d$};
\node[label={below: \fs ${\color{red} \begin{array}{c} {\color{black} y} \\ \shortuparrow \end{array} }$}] at (0,-1) {};
\end{tikzpicture}
\quad\quad\quad&
\begin{tikzpicture}[scale=0.6,>=stealth]
\draw[thick, smooth] (-1,0) -- (1,0);
\node at (-0.5,0) {$\r$}; \node at (0.5,0) {$\l$};
\node[label={left: \fs ${\color{red} \shortrightarrow} \ x$}] at (-1,0) {};
\draw[thick, smooth] (0,-1) -- (0,1);
\node at (0,-0.5) {$\d$}; \node at (0,0.5) {$\u$};
\node[label={below: \fs ${\color{red} \begin{array}{c} {\color{black} y} \\ \shortuparrow \end{array} }$}] at (0,-1) {};
\end{tikzpicture}
\\
\quad
$a_{+}(x,y)$
\quad\quad\quad&
\quad
$b_{+}(x,y)$
\quad\quad\quad&
\quad
$c_{+}(x,y)$
\\
\\
\begin{tikzpicture}[scale=0.6,>=stealth]
\draw[thick, smooth] (-1,0) -- (1,0);
\node at (-0.5,0) {$\l$}; \node at (0.5,0) {$\l$};
\node[label={left: \fs ${\color{red} \shortrightarrow} \ x$}] at (-1,0) {};
\draw[thick, smooth] (0,-1) -- (0,1);
\node at (0,-0.5) {$\d$}; \node at (0,0.5) {$\d$};
\node[label={below: \fs ${\color{red} \begin{array}{c} {\color{black} y} \\ \shortuparrow \end{array} }$}] at (0,-1) {};
\end{tikzpicture}
\quad\quad\quad&
\begin{tikzpicture}[scale=0.6,>=stealth]
\draw[thick, smooth] (-1,0) -- (1,0);
\node at (-0.5,0) {$\l$}; \node at (0.5,0) {$\l$};
\node[label={left: \fs ${\color{red} \shortrightarrow} \ x$}] at (-1,0) {};
\draw[thick, smooth] (0,-1) -- (0,1);
\node at (0,-0.5) {$\u$}; \node at (0,0.5) {$\u$};
\node[label={below: \fs ${\color{red} \begin{array}{c} {\color{black} y} \\ \shortuparrow \end{array} }$}] at (0,-1) {};
\end{tikzpicture}
\quad\quad\quad&
\begin{tikzpicture}[scale=0.6,>=stealth]
\draw[thick, smooth] (-1,0) -- (1,0);
\node at (-0.5,0) {$\l$}; \node at (0.5,0) {$\r$};
\node[label={left: \fs ${\color{red} \shortrightarrow} \ x$}] at (-1,0) {};
\draw[thick, smooth] (0,-1) -- (0,1);
\node at (0,-0.5) {$\u$}; \node at (0,0.5) {$\d$};
\node[label={below: \fs ${\color{red} \begin{array}{c} {\color{black} y} \\ \shortuparrow \end{array} }$}] at (0,-1) {};
\end{tikzpicture}
\\
\quad
$a_{-}(x,y)$
\quad\quad\quad &
\quad
$b_{-}(x,y)$
\quad\quad\quad &
\quad
$c_{-}(x,y)$
\end{tabular}
\caption{The vertices of the six-vertex model, with Boltzmann weights indicated beneath. The small red arrows indicate the orientation of the lines. In order to distinguish between $a$ and $b$ type vertices, the correct convention is to view every vertex such that its lines are oriented from south-west to north-east.}
\label{6vertices}
\end{figure}
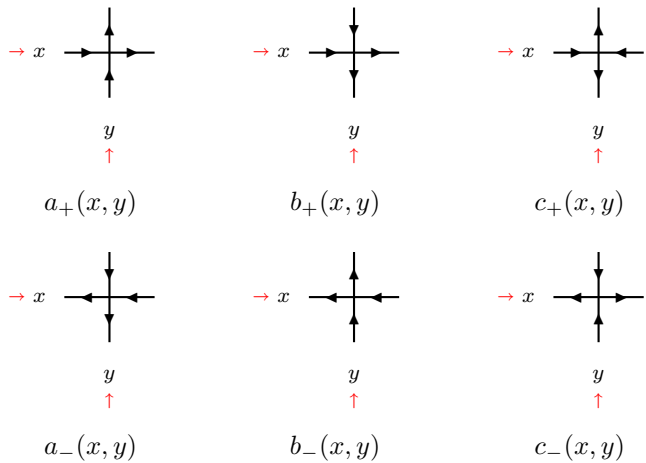
To each horizontal (respectively, vertical) line of the lattice one associates an orientation and a variable $x_i$ (respectively, $y_j$), which is its {\it rapidity}. The six types of vertex are assigned Boltzmann weights, which in this work are rational functions depending on the ratio $x/y$ of the rapidities incident on the vertex:
\begin{align}
\label{boltz}
\begin{array}{ll}
\displaystyle{
a_{+}(x,y)
=
\frac{1-t x/y}{1 - x/y},
}
&
\quad\quad
\displaystyle{
a_{-}(x,y)
=
\frac{1-t x/y}{1 - x/y},
}
\\
\displaystyle{ 
b_{+}(x,y) = 1,
}
&
\quad\quad
\displaystyle{ 
b_{-}(x,y) = t,
}
\phantom{\displaystyle{\frac{1-t x/y}{1 - x/y}}}
\\
\displaystyle{ 
c_{+}(x,y) 
= 
\frac{(1-t)}{1 - x/y},
}
&
\quad\quad
\displaystyle{ 
c_{-}(x,y) 
= 
\frac{(1-t) x/y}{1 - x/y}.
}
\end{array}
\end{align}
We point out that we have chosen a normalization in which the $b_{\pm}$ Boltzmann weights do not depend on the rapidity variables of the vertex. The essential feature of the weights thus constructed is that they satisfy the Yang--Baxter equation, which makes the six-vertex model integrable. In graphical form, the Yang--Baxter equation may be realised as in Figure \ref{fig:YB}.

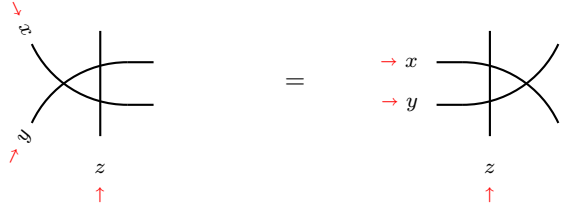
\begin{figure}[H]
\begin{tikzpicture}[scale=0.7]
\draw[thick,smooth] (0,1.5) arc (-155:-90:2);
\draw[thick,smooth] (0,0) arc (155:90:2);
\draw[thick, smooth] ({2*cos(25)},{1.5-(2-2*sin(25))})--({2*cos(25)+0.5},{1.5-(2-2*sin(25))});
\draw[thick, smooth] ({2*cos(25)},{(2-2*sin(25))})--({2*cos(25)+0.5},{(2-2*sin(25))});
\draw[thick,smooth] (1.3,-0.25)--(1.3,1.75);
\node at (5,0.75) {$=$};
\draw[thick,smooth] (10,1.5) arc (-25:-90:2);
\draw[thick,smooth] (10,0) arc (25:90:2);
\draw[thick, smooth] ({10-2*cos(25)},{1.5-(2-2*sin(25))})--({10-2*cos(25)-0.5},{1.5-(2-2*sin(25))});
\draw[thick, smooth] ({10-2*cos(25)},{(2-2*sin(25))})--({10-2*cos(25)-0.5},{(2-2*sin(25))});
\draw[thick,smooth] (8.7,-0.25)--(8.7,1.75);

\node[left,rotate=295] at (0,1.5) {\fs ${\color{red} \shortrightarrow} \ x$};
\node[left,rotate=-295] at (0,0) {\fs ${\color{red} \shortrightarrow} \ y$};

\node[label={below: \fs ${\color{red} \begin{array}{c}  {\color{black} z} \\ \shortuparrow \end{array} }$}] at (1.3,-0.25) {};

\node[label={left: \fs ${\color{red} \shortrightarrow} \ y$}] at ({10-2*cos(25)-0.5},{1.5-(2-2*sin(25))}) {};
\node[label={left: \fs ${\color{red} \shortrightarrow} \ x$}] at ({10-2*cos(25)-0.5},{(2-2*sin(25))}) {};
\node[label={below: \fs ${\color{red} \begin{array}{c} {\color{black} z} \\ \shortuparrow \end{array} }$}] at (8.7,-0.25) {};
\end{tikzpicture}
\caption{The Yang--Baxter equation. One makes a definite choice for the arrows on the six external edges (which is consistent and fixed on both sides of the equation) and sums over the possible arrow configurations on the three internal edges. In this way, the figure actually implies $2^6$ equations involving the Boltzmann weights \eref{boltz}.}
\label{fig:YB} 
\end{figure}
Most pertinent to this work, the Yang--Baxter equation implies that the partition functions that we study are symmetric functions in their rapidity variables, since it allows us to freely exchange a pair of horizontal or vertical lattice lines. 

One of the most fundamental quantities within the six-vertex model is the domain wall partition function (here and after abbreviated by DWPF). It was originally introduced by Korepin in \cite{kor}, and has played a key role in the study of scalar products and correlation functions in integrable quantum spin-chains (see \eg\ \cite{kbi}). The DWPF is most easily defined graphically, as in Figure \ref{DWPF}. 
\begin{figure}[H]
\begin{tikzpicture}[scale=0.6,>=stealth]
\foreach\x in {1,...,6}{
\draw[thick,smooth]
(0,\x) -- (7,\x);
\node at (0.5, \x) {\r};
\node at (6.5, \x) {\l};
}
\foreach\x in {1,...,6}
\node[label={left: \fs ${\color{red} \shortrightarrow} \ x_{\x}$}] at (0,7-\x) {};
\foreach\x in {1,...,6}{
\draw[thick,smooth]
(\x,0) -- (\x,7);
\node at (\x, 0.5) {\d};
\node at (\x,6.5) {\u};
}
\foreach\x in {1,...,6}
\node[label={below: \fs ${\color{red} \begin{array}{c} {\color{black} \b{y}_{\x}} \\ \shortuparrow \end{array} }$}] at (7-\x,0) {};
\end{tikzpicture}
\caption{$Z_{\rm ASM}$, the domain wall partition function in the case $n=6$. All external horizontal arrows point inwards, while external vertical arrows point outwards. We sum over all allowed configurations of the internal edges. For convenience, we have reciprocated the variables on the vertical lines, and set them to $\b{y}_j = 1/y_j$.}
\label{DWPF} 
\end{figure}
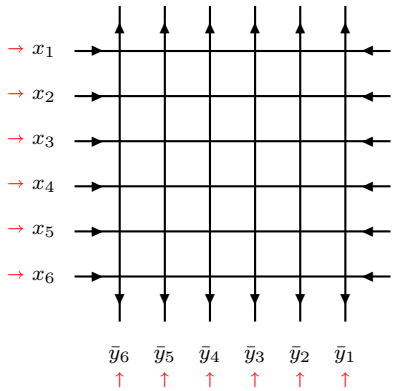
The DWPF was evaluated in determinant form by Izergin in \cite{ize}. With the Boltzmann weights as defined in \eref{boltz}, the determinant expression is
\begin{align}
\label{ize-det}
Z_{\rm ASM}(x_1,\dots,x_n ; y_1,\dots, y_n ; t)
=
\frac{\prod_{i,j=1}^{n} (1-t x_i y_j)}
{\prod_{1 \leq i<j \leq n} (x_i - x_j) (y_i - y_j)}
\det
\left[
\frac{(1-t)}{(1-x_i y_j)(1-t x_i y_j)}
\right]_{1\leq i,j \leq n}.
\end{align}
We will not elaborate on the proof of \eref{ize-det}, but it can be proved using the well known Izergin--Korepin technique, whereby one writes down a set of conditions which uniquely determine $Z_{\rm ASM}$ and shows that \eref{ize-det} obeys all such conditions.

As was noticed in \cite{kup1}, configurations of the six-vertex model with domain wall boundary conditions are in one-to-one correspondence with alternating sign matrices (ASMs). Using this fact, and evaluating the determinant expression \eref{ize-det} in a certain homogeneous limit, Kuperberg was able to obtain an elegant proof of the ASM conjecture \cite{mrr}. In this way, one can think of \eref{ize-det} as being a multivariate generating function of ASMs, and it is for that reason that we use the notation $Z_{\rm ASM}$.

Finally, we point out that the DWPF \eref{ize-det} is identically equal to the right hand side of equation \eref{knw-id} (this fact had already been observed by Warnaar in \cite{war}). Our observation is that 
\eref{pp-ASM-gs} relates a generating series of path-weighted plane partitions with a generating series of ASMs, providing a new and potentially interesting link between these combinatorial objects. 

\subsection{Partition functions on rectangular domains}
\label{ssec:p-dwpf}

Up until now, we restricted our attention to the case where the cardinalities of the sets 
$\{x_1,\dots,x_n\}$ and $\{y_1,\dots,y_n\}$ in \eref{knw-id} are equal. Clearly we are at liberty to take a subset of either of these to be zero, choosing for example $x_{m+1} = \cdots = x_n = 0$, for some $m<n$. Due to the stability property $P_{\lambda}(x_1,\dots,x_{n-1},0;t) = P_{\lambda}(x_1,\dots,x_{n-1};t)$ of 
Hall--Littlewood polynomials, taking this limit transforms the left hand side of \eref{knw-id} to the case of unequal cardinalities. To calculate the resulting right hand side, we observe that the Izergin determinant may be written in the following form:
\begin{align}
\label{row-op-det}
\det\left[
\frac{(1-t)}{(1-x_i y_j)(1-tx_i y_j)}
\right]_{1\leq i,j \leq n}
=
\prod_{m+1 \leq i<j \leq n} (x_i - x_j)
\det\Big[
\mathcal{A}_{i,j}
\Big]_{1\leq i,j \leq n},
\end{align}
where the entries of the final determinant are given by
\begin{align*}
\mathcal{A}_{i,j}
=
\left\{
\begin{array}{ll}
\mathcal{A}_{i,j}(x_i;y_j)
=
\displaystyle{
\frac{(1-t)}{(1-x_i y_j)(1-t x_i y_j)}
},
\quad
&  
1 \leq i \leq m,
\\ \\ 
\mathcal{A}_{i,j}(x_i,\dots,x_n;y_j)
=
\displaystyle{\sum_{k=0}^{\infty}}
(1-t^{k+1})h_{k+i-n}(x_i,\dots,x_n) y_j^k,
\quad
& 
m+1 \leq i \leq n,
\end{array}
\right.
\end{align*}
with $h_{k+i-n}(x_i,\dots,x_n)$ denoting a complete symmetric function (see Section 2, Chapter I of \cite{mac}). The derivation of \eref{row-op-det} is by expanding the entries of the original determinant as formal power series, and using row operations in conjunction with the identity
\begin{align*} 
h_k(x_1,\dots,x_l,x_m) 
-
h_k(x_1,\dots,x_l,x_n)
=
(x_m - x_n)
h_{k-1}(x_1,\dots,x_l,x_m,x_n)
\end{align*}
for complete symmetric functions. Writing the determinant in this way, $x_i \rightarrow 0$ for all 
$m+1 \leq i \leq n$ is no longer a singular limit of equation \eref{knw-id}. We thus obtain
\begin{multline}
\label{knw-pdwpf}
\sum_{\lambda}
\prod_{j=1}^{n-\ell(\lambda)}
(1-t^j)
b_{\lambda}(t)
P_{\lambda}(x_1,\dots,x_m;t)
P_{\lambda}(y_1,\dots,y_n;t)
=
\\
\frac{\prod_{i=1}^{n-m} (1-t^i) \prod_{i=1}^{m} \prod_{j=1}^{n} (1-tx_i y_j)}
{\prod_{i=1}^{m} (x_i^{n-m}) \Delta(x)_m \Delta(y)_n}
\det\Big[
\mathcal{A}^{\circ}_{i,j}
\Big]_{1 \leq i,j \leq n},
\end{multline}
where the entries of the determinant are given by
\begin{align*}
\mathcal{A}^{\circ}_{i,j}
=
\frac{(1-t)}{(1-x_i y_j)(1-t x_i y_j)},
\ \  
1 \leq i \leq m,
\quad\quad\quad\quad
\mathcal{A}^{\circ}_{i,j}
=
y_j^{n-i},
\ \ 
m+1 \leq i \leq n.
\end{align*}
It is natural to ask whether the right hand side of the new identity \eref{knw-pdwpf} has an interpretation within the six-vertex model. To answer this question, we consider the form of the Boltzmann weights \eref{boltz} at $x=0$. They simplify to
\begin{align} 
\label{triv-boltz}
a_{\pm} = 1,\quad b_{+} = 1,\quad b_{-} = t,\quad c_{+} = 1-t,\quad c_{-} = 0.
\end{align}
The vanishing of the $c_{-}$ vertex plays a crucial role in studying the DWPF in this limit. It means that when we set $x_{m+1} = \cdots = x_n = 0$, no $c_{-}$ vertices are allowed in the bottom $n-m$ rows of Figure \ref{DWPF}, and that these rows contribute only a multiplicative factor to the partition function. The bottom $n-m$ rows can effectively be deleted from the lattice, at the expense of the factor 
$\prod_{i=1}^{n-m} (1-t^i)$, while the lower external edges of the resulting lattice are summed over all possible arrow configurations. For more details on this limiting procedure, and in particular for the derivation of the overall multiplicative factor, we refer the reader to \cite{fw2}. The resulting lattice expression was called a {\it partial} domain wall partition function (pDWPF) in \cite{fw2}, and is illustrated in Figure \ref{p-DWPF}. We denote the pDWPF by $Z_{\rm ASM}(x_1,\dots,x_m;y_1,\dots,y_n;t)$.

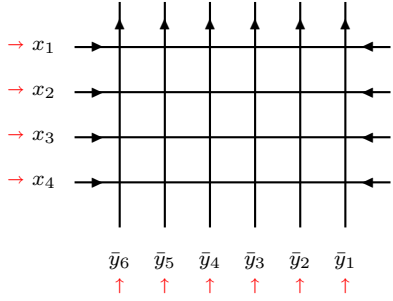
\begin{figure}[H]
\begin{tikzpicture}[scale=0.6,>=stealth]
\foreach\x in {1,...,4}{
\draw[thick, smooth] (0,\x) -- (7,\x);
\node at (0.5, \x) {\r};
\node at (6.5, \x) {\l};
}
\foreach\x in {1,...,4}
\node[label={left: \fs ${\color{red} \shortrightarrow} \ x_{\x}$}] at (0,5-\x) {};
\foreach\x in {1,...,6}{
\draw[thick, smooth] (\x,0) -- (\x,5);
\node at (\x, 4.5) {\u};
}
\foreach\x in {1,...,6}
\node[label={below: \fs ${\color{red} \begin{array}{c} {\color{black} \b{y}_{\x}} \\ \shortuparrow \end{array}}$}] at (7-\x,0) {};
\end{tikzpicture}
\caption{The partial domain wall partition function 
$Z_{\rm ASM}(x_1,\dots,x_m;y_1,\dots,y_n;t)$, which arises from setting $x_{m+1} = \cdots = x_n = 0$ in Figure \ref{DWPF}. In the example shown, $m=4$, $n=6$. The lower external edges of the lattice are to be summed over all possible arrow configurations, which is why we leave them blank. In this example, the sum is over the 
$\binom{6}{2}$ ways of assigning 2 upward and 4 downward arrows to the 6 edges.}
\label{p-DWPF} 
\end{figure}
Cancelling the factor $\prod_{i=1}^{n-m} (1-t^i)$ from the right hand side of \eref{knw-pdwpf}, it is equal to the pDWPF: 
\begin{multline*}
\sum_{\lambda}
\prod_{j=1}^{m-\ell(\lambda)}
(1-t^{j+n-m})
b_{\lambda}(t)
P_{\lambda}(x_1,\dots,x_m;t)
P_{\lambda}(y_1,\dots,y_n;t)
\\
=
\frac{\prod_{i=1}^{m} \prod_{j=1}^{n} (1-tx_i y_j)}
{\prod_{i=1}^{m} (x_i^{n-m}) \Delta(x)_m \Delta(y)_n}
\det\Big[
\mathcal{A}^{\circ}_{i,j}
\Big]_{1 \leq i,j \leq n}
=
Z_{\rm ASM}(x_1,\dots,x_m; y_1,\dots,y_n; t),
\end{multline*}
where we have used the implicit constraint $\ell(\lambda) \leq m$ on the summation over $\lambda$ to remove the common factor $\prod_{i=1}^{n-m} (1-t^i)$ from both sides of \eref{knw-pdwpf}.


\section{Refined symplectic Cauchy identities and U-turn alternating sign matrices}
\label{sec:UASM}

\subsection{A further Cauchy determinant} Let $x_1,\dots,x_n$ and $y_1,\dots,y_n$ be indeterminates. The following determinantal formula is also due to Cauchy:
\begin{align}
\det\left[ \frac{1}{(1-x_i y_j)(1-x_i \b{y}_j)} \right]_{1\leq i,j \leq n}
=
\frac{
\Delta(x)_n \Delta(y)_n 
\prod_{1 \leq i<j \leq n} 
(1-x_i x_j)(1-\b{y}_i \b{y}_j)
}
{\prod_{i,j=1}^n (1-x_i y_j) (1-x_i \b{y}_j)}.
\label{further-cauch-det}
\end{align}
One can prove this by factor exhaustion.

\subsection{Weyl formula for symplectic characters}

Symplectic characters $sp_{\lambda}(x_1,\b{x}_1,\dots,x_n,\b{x}_n)$ are the characters of irreducible representations of $Sp(2n)$. The Weyl character formula gives rise to the determinant expression:
\begin{align*}
sp_{\lambda}(x_1,\b{x}_1,\dots,x_n,\b{x}_n)
=
\frac{
\det\left[ x_i^{\lambda_j-j+n+1} - \b{x}_i^{\lambda_j-j+n+1} \right]_{1 \leq i,j \leq n}
}
{
\prod_{i=1}^{n} (x_i - \b{x}_i)
\prod_{1 \leq i<j \leq n} (x_i - x_j) (1 - \b{x}_i \b{x}_j)
}.
\end{align*}

\subsection{Tableau formula for symplectic characters}

In direct analogy with Schur polynomials, a symplectic character can be defined as a sum over tableaux which obey certain properties \cite{sun,kes}. A {\it symplectic tableau} of shape $\lambda$ is an assignment of one symbol $\{1,\b{1},\dots,n,\b{n}\}$ to each box of the Young diagram $\lambda$, subject to the following rules:
\begin{enumerate}[{\bf 1.}]
\item The symbols have the ordering $1 < \b{1} < \cdots < n < \b{n}$.
\item The entries in $\lambda$ increase weakly along each row and strictly down each column.
\item All entries in row $i$ of $\lambda$ are at least $i$. 
\end{enumerate}
Notice that conditions {\bf 1} and {\bf 2} are the usual rules of a semi-standard Young tableau (albeit in terms of a different alphabet of symbols). It is condition {\bf 3} which 
non-trivially distinguishes a symplectic tableau from an ordinary semi-standard tableau.

With this definition, $sp_{\lambda}(x_1,\b{x}_1,\dots,x_n,\b{x}_n)$ can be alternatively defined as a weighted sum over symplectic tableaux $\overline{T}$ of shape $\lambda$:
\begin{align}
sp_{\lambda}(x_1,\b{x}_1,\dots,x_n,\b{x}_n)
=
\sum_{\overline{T}}\  
\prod_{k=1}^{n}
x_k^{\#(k)-\#(\b{k})},
\label{symp-tab}
\end{align}
where $\#(k)$ and $\#(\b{k})$ count the instances of $k$ and $\b{k}$ in $\overline{T}$, respectively.

\subsection{Symplectic tableaux as restricted interlacing sequences}

As with ordinary SSYT, any symplectic tableau $\overline{T}$ can be decomposed into a sequence of interlacing partitions, but with a length constraint imposed on every second partition:
\begin{align*}
\overline{T}
=
\{
\emptyset \equiv \b{\lambda}^{(0)} \prec \lambda^{(1)} \prec \b{\lambda}^{(1)} \prec 
\cdots
\prec \lambda^{(n)} \prec \b{\lambda}^{(n)} \equiv \lambda\ 
|\ 
\ell(\b{\lambda}^{(i)}) \leq i
\}.
\end{align*}
The constraint $\ell(\b{\lambda}^{(i)}) \leq i$ is a direct consequence of property {\bf 3} of symplectic tableaux.

For any integer $k$, the number of instances of $k$ in $\overline{T}$ is given by 
$|\lambda^{(k)}| - |\b{\lambda}^{(k-1)}|$, while that of $\b{k}$ is given by 
$|\b{\lambda}^{(k)}| - |\lambda^{(k)}|$. Combining this fact with equation \eref{symp-tab}, the symplectic character can be expressed as
\begin{align}
\label{symp-int}
sp_{\lambda}(x_1,\b{x}_1,\dots,x_n,\b{x}_n)
&=
\sum_{\overline{T}}
\ \ 
\prod_{i=1}^{n} x_i^{|\lambda^{(i)}|-|\b{\lambda}^{(i-1)}|}
\prod_{j=1}^{n} x_j^{|\lambda^{(j)}|-|\b{\lambda}^{(j)}|}
= 
\sum_{\overline{T}}
\ \ 
\prod_{i=1}^{n} x_i^{2|\lambda^{(i)}|-|\b{\lambda}^{(i)}|-|\b{\lambda}^{(i-1)}|}. 
\end{align}

\subsection{Cauchy identity for symplectic characters and associated plane partitions}

Motivated by the tableau definition \eref{symp-tab} of symplectic characters, in this section we define a class of plane partitions which (to the best of our knowledge) have not been previously studied in the literature. We take as our starting point the Cauchy identity for symplectic characters \cite{sun}:
\begin{align}
\sum_{\lambda}
s_{\lambda}(x_1,\dots,x_m)
sp_{\lambda}(y_1,\b{y}_1,\dots,y_n,\b{y}_n)
=
\frac{
\prod_{1\leq i < j \leq m} (1-x_i x_j)
}
{
\prod_{i=1}^{m}
\prod_{j=1}^{n}
(1-x_i y_j)(1-x_i \b{y}_j)
}
\label{symp-cauch}
\end{align}
where $m \leq n$, and wish to view the left hand side of \eref{symp-cauch} as a generating series for plane partitions. With that in mind, we introduce the following set of plane partitions: 
\begin{align*}
\bm{\overline{\pi}}_{m,2n}
=
\{ 
\emptyset \equiv \lambda^{(0)}
\prec \lambda^{(1)} \prec \cdots \prec \lambda^{(m)}
\equiv
\b{\mu}^{(n)} \succ \mu^{(n)} \succ \cdots \succ \b{\mu}^{(1)} \succ \mu^{(1)} \succ 
\b{\mu}^{(0)} \equiv \emptyset
\}
\end{align*}
subject to the additional constraint $\ell(\b{\mu}^{(i)}) \leq i$, for all $1 \leq i \leq n$. We refer to these as {\it symplectic plane partitions.} Using the tableaux formulae \eref{schur-int} and \eref{symp-int}, the Cauchy identity \eref{symp-cauch} can then be rephrased as
\begin{align}
\sum_{\pi \in \bm{\overline{\pi}}_{m,2n}}
\prod_{i=1}^{m} x_i^{|\lambda^{(i)}|-|\lambda^{(i-1)}|}
\prod_{j=1}^{n} y_j^{2 |\mu^{(j)}|-|\b{\mu}^{(j)}|-|\b{\mu}^{(j-1)}|}
=
\frac{
\prod_{1\leq i < j \leq m} (1-x_i x_j)
}
{
\prod_{i=1}^{m}
\prod_{j=1}^{n}
(1-x_i y_j)(1-x_i \b{y}_j)
}
\label{symp-cauch-pp}
\end{align}
which is evidently a generating function for symplectic plane partitions.

With a view to obtaining an analogue of volume-weighted plane partitions \eref{vol-pp} in the symplectic case, one can consider $q$-specializations of the parameters 
$(x_1,\dots,x_m)$ and $(y_1,\dots,y_n)$. However, care must be taken in doing so, since \eref{symp-cauch-pp} may become singular (which means that for a given weight there will be infinitely many plane partitions with that weight). Choosing $x_i = q^{m-i+3/2}$ for all $1\leq i \leq m$ and $y_j = q^{1/2}$ for all 
$1\leq j \leq n$, we reduce \eref{symp-cauch-pp} to
\begin{align}
\label{symp-pp-vol}
\sum_{\pi \in \bm{\overline{\pi}}_{m,2n}}
q^{|\pi_{\leq}|} q^{|\pi_{>}^{o}|-|\pi_{>}^{e}|} 
=
\frac{
\prod_{1\leq i < j \leq m} (1-q^{i+j+1})
}
{
\prod_{i=1}^{m} (1-q^i)^n (1-q^{i+1})^n
}
\end{align}
where we have defined $\pi_{<}$ and $\pi_{>}$ to be the left and right halves of the symplectic plane partition $\pi$, with the main diagonal {\it omitted.} $\pi_{\leq}$ and $\pi_{\geq}$ are the same, but with the main diagonal {\it included.} Similarly $\pi^{e}$ and $\pi^{o}$ are the set of even and odd slices of $\pi$. Intersections of these sets are denoted in the obvious way, \eg\ $\pi_{>}^{e}$ is the intersection of the right half and the even slices of $\pi$, excluding the main diagonal. For an illustration of the weighting \eref{symp-pp-vol} on a typical symplectic plane partition, see Figure \ref{fig:sympa}.

\begin{figure}[H]
\includegraphics[scale=0.4]{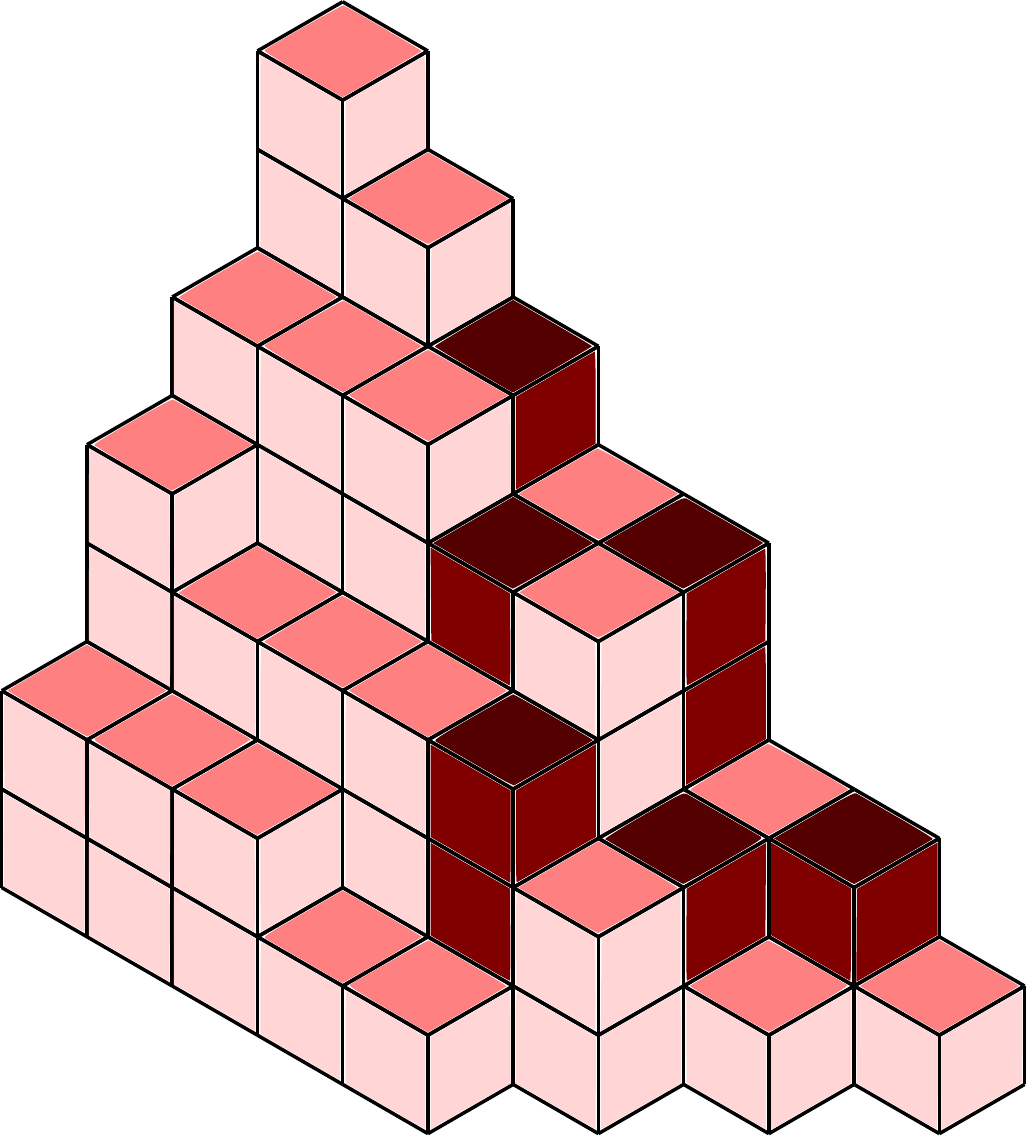}
\caption{A symplectic plane partition $\pi \in \bm{\overline{\pi}}_{4,8}$. The lightly shaded cubes contribute a positive value to the volume, while the darkly shaded cubes contribute a negative value.}
\label{fig:sympa}
\end{figure}

\subsection{A $t$-analogue of the symplectic Cauchy identity}

A natural question is whether we can $t$-deform the symplectic Cauchy identity \eref{symp-cauch}. The answer is yes, and one natural way\footnote{In a different direction, Hamel and King \cite{hk} provide another $t$-generalization of symplectic characters in terms of shifted symplectic tableaux and prove a bijection between shifted tableaux of shape $\mu$ and $\mu$-UASMs.} (at least formulaically) goes through a class of symmetric polynomials defined by Rains in \cite{rai}. The fully general definition of these functions is beyond the scope of this work (see \cite{rai}, Sections 6 and 7), but one easy way to define them is precisely via the associated Cauchy identity. We will not be concerned with the most general case, so we introduce five parameters $t,t_0,t_1,t_2,t_3$. The polynomials in question are again indexed by partitions, so if $\lambda$ is a partition with at most $\ell$ parts, we define $\tilde{K}_{\lambda} (z_1,\dots,z_{\ell};t_0,\dots,t_3;t)$ to be the coefficient of the Hall--Littlewood 
$Q$-polynomial 
$$
Q_{\lambda}(x_1,\dots,x_m;t):= 
\prod_{i=1}^{\infty} \prod_{j=1}^{m_i(\lambda)}(1-t^j) 
P_{\lambda}(x_1,\dots,x_m;t)
$$ 
in the expansion
\begin{multline}
\label{Ktilde-cauchy}
\sum_{\lambda} 
\prod_{i=1}^{\infty}
\prod_{j=1}^{m_i(\lambda)}(1-t^j) 
P_{\lambda}(x_1,\dots,x_m;t)
\tilde{K}_{\lambda} (z_1,\dots,z_{\ell};t_0,\dots,t_3;t)  
= 
\\
\prod_{i=1}^{m}
\prod_{j=1}^{\ell}
\frac{1-t x_i z_j}{1-x_i z_j} 
\prod_{1 \leq i<j \leq m} 
\frac{1-x_i x_j}{1-t x_i x_j} 
\prod_{i=1}^{m} 
\frac{(1-t_0 x_i)(1-t_1 x_i)(1-t_2 x_i)(1-t_3 x_i)}{1-t x_i^2}.
\end{multline}
By setting $t = t_0 = t_1 = t_2 = t_3 = 0$, choosing $\ell = 2n$ and
\begin{align}
\label{COV} 
z_{2i-1} = y_i,\quad z_{2i} = 1/y_i = \b{y}_i,\quad 1 \leq i \leq n,
\end{align} 
the above Cauchy identity \eref{Ktilde-cauchy} becomes the symplectic one \eref{symp-cauch} and hence the $\tilde{K}_{\lambda} = sp_{\lambda}$. In that sense, $\tilde{K}_{\lambda}$ is a lift of the hyperoctahedrally symmetric symplectic characters to the land of (ordinary $S_n$) symmetric functions. We refer the reader to \cite{rai} for more on the theory.

We are not aware of a tableau rule (or definition) for $\tilde{K}_{\lambda}$, nor of any other interesting properties that are not already listed in \cite{rai}, but we hope that the above Cauchy identity (with its nice factorized kernel) is more than a mere curiosity, seeing as it generalizes symplectic characters (and, in a different limit, orthogonal characters also). 

\subsection{$BC_n$-symmetric Hall--Littlewood polynomials}

In order to find an analogue of Theorem \ref{thm2} that generalizes equation \eref{symp-cauch}, we now introduce Hall--Littlewood polynomials of type $BC_n$. They can be seen as the limit $q \to 0$ of Koornwinder polynomials (see \cite{rai} for explicit formulas for Koornwinder polynomials). However, we will adopt Venkateswaran's point of view \cite{ven} and alternatively define them via a sum over the hyperoctahedral group in much the same way as the usual ($A_n$-symmetric) Hall--Littlewood polynomials can be defined as a sum over the symmetric group (see Section \ref{HL-def}).

The hyperoctahedral group $W(BC_n)$ on $n$ symbols is defined for our purposes as the group of $2^n n!$ signed permutations: $W(BC_n) = S_n \rtimes \mathbb{Z}_2^n$. On a set of indeterminates $\{x_1,\dots,x_n\}$ it acts by permuting some elements and possibly inverting some elements (that is, sending $x \to 1/x = \b{x}$). 

$BC_n$-symmetric Hall--Littlewood polynomials (in $n$ variables) are a special class of multivariate (hence indexed by partitions) orthogonal Laurent polynomials, symmetric under the aforementioned action of $W(BC_n)$. They depend {\it a priori} on five parameters $t, t_0, t_1, t_2, t_3$ (the same parameters descending from the theory of Koornwinder polynomials) but for our purposes we set $t_0 = t_1 = t_2 = t_3 = 0$. Following Venkateswaran \cite{ven}, let $\lambda$ be a partition with $n$ or fewer parts, and define the 
$BC_n$-symmetric Hall--Littlewood polynomial indexed by $\lambda$ as
\begin{align} 
\label{BCHL-def}
K_{\lambda}(x_1,\b{x}_1,\dots,x_n,\b{x}_n;t)
:=
\frac{1}{v_{\lambda}(t)}
\sum_{\omega \in W(BC_n)}
\omega\left(
\prod_{i=1}^{n}
\frac{x_i^{\lambda_i}}{(1-\b{x}_i^2)}
\prod_{1 \leq i<j \leq n}
\frac{(x_i-t x_j)(1-t \b{x}_i \b{x}_j )}{(x_i - x_j)(1 - \b{x}_i \b{x}_j)}
\right).
\end{align}
It is not difficult to check that this definition indeed leads to Laurent polynomials (a slight modification of the usual argument, see \eg\ \cite{mac}, applies). They are clearly $BC_n$-symmetric and we refer the reader to \cite{ven} for the proof of orthogonality and other theorems related to them (in particular, for the proof that one indeed obtains Koornwinder polynomials at $q=0$ in this way). 

We remark that when $t=0$, the $BC_n$-symmetric Hall--Littlewood polynomials $K_{\lambda}$ become symplectic characters $sp_{\lambda}$. However it is important to point out that for generic $t$, the polynomials \eref{BCHL-def} {\it do not} coincide with the $\tilde{K}_{\lambda}$ polynomials of Rains with $t_0=t_1=t_2=t_3=0$ and with variables chosen according to \eref{COV}.

Unlike the case of ordinary ($A_n$-symmetric) Hall--Littlewood polynomials, and indeed unlike symplectic characters, the authors are not aware of an explicit tableau-like formula for $BC_n$-symmetric Hall--Littlewood polynomials. Thus we cannot connect these polynomials with any types of plane partitions at the moment.

\subsection{Refined symplectic Cauchy identities}

We now state a theorem and a conjecture which, in the $BC_n$-symmetric case, mirror Theorem \ref{thm1} and Theorem \ref{thm2}. The theorem involves the expansion of a certain determinant (the partition function for UASMs under certain weights) in Schur polynomials and symplectic characters. The conjecture involves the 
$BC_n$-symmetric Hall--Littlewood polynomials defined above. In both cases, the symplectic Cauchy identity \eref{symp-cauch} is recovered as the special case $t=0$, which is easily deduced using the determinant factorization \eref{further-cauch-det}. 

\begin{thm}
\label{thm3}
\begin{multline}
\label{s-uasm}
\sum_{\lambda}
\prod_{i=1}^{n}
(1-t^{\lambda_i-i+n+1})
s_{\lambda}(x_1,\dots,x_n)
sp_{\lambda}(y_1,\b{y}_1,\dots,y_n,\b{y}_n)
\\
=
\frac{\prod_{i=1}^{n} (1-t x_i^2)}
{\Delta(x)_n \Delta(y)_n \prod_{1\leq i < j \leq n} (1-\b{y}_i \b{y}_j)}
\det
\left[
\frac{(1-t)}
{(1- x_i y_j)(1- t x_i y_j)(1- x_i \b{y}_j)(1- t x_i \b{y}_j)}
\right]_{1\leq i,j \leq n}.
\end{multline}

\end{thm}

\begin{proof}
Similarly to the case of Theorem \ref{thm1}, it is possible to prove this identity by acting on the $m=n$ symplectic Cauchy identity \eref{symp-cauch} with the generating series \eref{do-gs}, evaluated at $z=-t$ and $q=t$. We will again pursue a simpler proof, which directly parallels our proof of Theorem \ref{thm1}. We take the determinant in \eref{s-uasm}, and (after making some trivial manipulations) treat its entries as formal power series:
\begin{align*}
&
\prod_{i=1}^{n}
(1-t x_i^2)
\det
\left[
\frac{(1-t)}
{(1- x_i y_j)(1- t x_i y_j)(1- x_i \b{y}_j)(1- t x_i \b{y}_j)}
\right]_{1\leq i,j \leq n}
\\
&=
\prod_{i=1}^{n}
\frac{1}{(y_i - \b{y}_i)}
\det
\left[
\frac{(1-t)y_j}{(1- x_i y_j)(1- t x_i y_j)}
-
\frac{(1-t)\b{y}_j}{(1- x_i \b{y}_j)(1- t x_i \b{y}_j)}
\right]_{1 \leq i,j \leq n}
\\
&=
\prod_{i=1}^{n}
\frac{1}{(y_i - \b{y}_i)}
\det
\left[
\sum_{k=0}^{\infty}
(1-t^{k+1}) x_i^k (y_j^{k+1} - \b{y}_j^{k+1})
\right]_{1 \leq i,j \leq n}
\\
&=
\prod_{i=1}^{n}
\frac{1}{(y_i - \b{y}_i)}
\sum_{k_1 > \cdots > k_n \geq 0}
\det\left[
x_i^{k_j}
\right]_{1\leq i,j \leq n}
\det\left[
y_j^{k_i+1} - \b{y}_j^{k_i+1}
\right]_{1\leq i,j \leq n}
\prod_{i=1}^{n}
(1-t^{k_i+1}).
\end{align*}
The proof is completed by making the previous change of summation indices 
$k_i = \lambda_i - i + n$, and dividing by the factor 
$\Delta(x)_n \Delta(y)_n \prod_{1 \leq i<j \leq n} (1-\b{y}_i \b{y}_j)$.
\end{proof}

The following conjecture is surprising, and has been checked in Mathematica for small enough partitions. More precisely, we have checked that if we expand the right hand side determinant in 
Hall--Littlewood polynomials and take their coefficients corresponding to small enough partitions, those coefficients equal $BC_n$-symmetric Hall--Littlewood polynomials (for the same partition). 

\begin{conj}
\label{conj1}
\begin{multline}
\label{uasm-conj}
\sum_{\lambda}
\prod_{i=0}^{\infty}
\prod_{j=1}^{m_i(\lambda)}
(1-t^j)
P_{\lambda}(x_1,\dots,x_n;t)
K_{\lambda}(y_1,\b{y}_1,\dots,y_n,\b{y}_n;t)
\\
=
\frac{
\prod_{i,j=1}^{n} 
(1-t x_i y_j) 
(1-t x_i \b{y}_j)
}
{\prod_{1\leq i<j \leq n} (x_i-x_j) (y_i-y_j) (1 - t x_i x_j) (1 - \b{y}_i \b{y}_j) }
\det\left[
\frac{(1-t)}{(1- x_i y_j)(1-t x_i y_j)(1- x_i \b{y}_j)(1-t x_i \b{y}_j)}
\right]_{1\leq i,j \leq n}.
\end{multline}
\end{conj}

\subsection{Refined symplectic plane partitions}

The combinatorial meaning of equation \eref{s-uasm} is quite analogous to that of \eref{s-cauchy-refine}, in Section \ref{sec:refined-cauchy}. Namely, it is a $t$-refinement of the generating series \eref{symp-cauch-pp}, which assigns a $t$-dependent weight to the central slice of the symplectic plane partition:
\begin{multline}
\label{symp-pp-UASM}
\sum_{\pi \in \bm{\overline{\pi}}_{n,2n}}
\prod_{i=1}^{n} (1-t^{\pi(i,i)-i+n+1})
\prod_{i=1}^{n} 
x_i^{|\lambda^{(i)}|-|\lambda^{(i-1)}|}
y_i^{2 |\mu^{(i)}|-|\b{\mu}^{(i)}|-|\b{\mu}^{(i-1)}|}
=
\\
\frac{\prod_{i=1}^{n} (1-t x_i^2)}
{\Delta(x)_n \Delta(y)_n \prod_{1\leq i < j \leq n} (1-\b{y}_i \b{y}_j)}
\det
\left[
\frac{(1-t)}
{(1- x_i y_j)(1- t x_i y_j)(1- x_i \b{y}_j)(1- t x_i \b{y}_j)}
\right]_{1\leq i,j \leq n}.
\end{multline}
By setting $t=0$, one recovers the original identity \eref{symp-cauch-pp} at $m=n$. Unfortunately, we have not been able to find an appropriate combinatorial interpretation for equation \eref{uasm-conj}. The main obstacle here is the absence of a tableau rule for the $BC_n$-symmetric Hall--Littlewood polynomials, which is an essential requirement for translating from Cauchy-type identities to plane partitions.

\subsection{Six-vertex model on UASM lattice}
\label{sec:U-vert}

We return our attention to the six-vertex model, but on a different domain to the one considered in Section \ref{sec:6v}. All previous conventions apply regarding the vertices of the model and their Boltzmann weights, but we introduce two additional (boundary) U-turn vertices, shown in Figure \ref{fig:U}.

\begin{figure}[H]
\begin{tabular}{cc}
\begin{tikzpicture}[scale=0.6]
\draw[thick, smooth] ({2*sqrt(2) + 3 + 4},2) arc (90:-90:1);
\draw[thick, smooth] ({2*sqrt(2) + 3 + 3.5},2)--({2*sqrt(2) + 3 + 4},2);
\draw[thick, smooth] ({2*sqrt(2) + 3 + 3.5},0)--({2*sqrt(2) + 3 + 4},0);
\node at ({2*sqrt(2) + 3 + 5},1) {$\bullet$};
\node at ({2*sqrt(2) + 3 + 4},0) {\r};
\node at ({2*sqrt(2) + 3 + 4},2) {\l};
\node[label={left: \fs ${\color{red} \shortrightarrow} \ x$}] at ({2*sqrt(2) + 6.5},0) {};
\node[label={left: \fs ${\color{red} \shortleftarrow} \ \b{x}$}] at ({2*sqrt(2) + 6.5},2) {};
\end{tikzpicture}
\quad\quad\quad&
\begin{tikzpicture}[scale=0.6]
\draw[thick, smooth] ({2*sqrt(2) + 3 + 4},2) arc (90:-90:1);
\draw[thick, smooth] ({2*sqrt(2) + 3 + 3.5},2)--({2*sqrt(2) + 3 + 4},2);
\draw[thick, smooth] ({2*sqrt(2) + 3 + 3.5},0)--({2*sqrt(2) + 3 + 4},0);
\node at ({2*sqrt(2) + 3 + 5},1) {$\bullet$};
\node at ({2*sqrt(2) + 3 + 4},0) {\l};
\node at ({2*sqrt(2) + 3 + 4},2) {\r};
\node[label={left: \fs ${\color{red} \shortrightarrow} \ x$}] at ({2*sqrt(2) + 6.5},0) {};
\node[label={left: \fs ${\color{red} \shortleftarrow} \ \b{x}$}] at ({2*sqrt(2) + 6.5},2) {};
\end{tikzpicture}
\\ \\
$k_{+}(x)$
\quad\quad\quad&
$k_{-}(x)$
\end{tabular}
\caption{The U-turn vertices with their Boltzmann weights indicated underneath.}
\label{fig:U}
\end{figure}
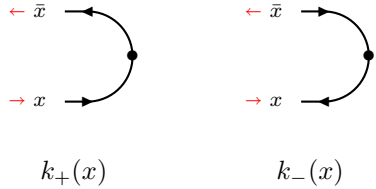
In contrast to the regular vertices of the model, the U-turn vertices consist of a single line, and a single associated rapidity variable. They are also assigned Boltzmann weights, which in the most general setting (see \eg\ \cite{kup2}) can be functions depending both on the rapidity $x$ and another free parameter $b$. We are interested in a degenerate limit of this more general case, in which the Boltzmann weights become equal:
\begin{align}
\label{U-vert}
k_{+}(x) = k_{-}(x) = \frac{1}{1-x^2}.
\end{align}
The reason that we adopt this choice for the boundary weights (including their normalization) is so that we obtain exact agreement between the quantity appearing on the right hand side of Conjecture \ref{conj1}, and the partition function that we subsequently discuss. Together with the weights assigned to the bulk vertices \eref{boltz}, the boundary weights \eref{U-vert} satisfy Sklyanin's reflection equation \cite{skl}. The graphical version of the reflection equation is given in Figure \ref{fig:refl}.

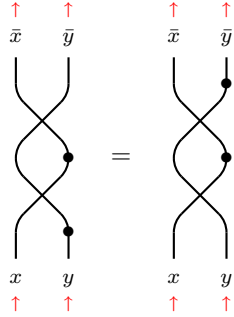
\begin{figure}[H]
\begin{tikzpicture}[scale=0.35]

\node[label={below: \fs ${\color{red} \begin{array}{c} {\color{black} x} \\ \shortuparrow \end{array} }$}] at (1,0.5) {};
\node[label={below: \fs ${\color{red} \begin{array}{c} {\color{black} y} \\ \shortuparrow \end{array} }$}] at (3,0.5) {};
\draw[thick,smooth] (1,0)--(1,1); 
\draw[thick,smooth] (3,0)--(3,1);
\draw[thick,smooth] (1,1) arc (180:135:1);
\draw[thick,smooth] (3,1) arc (0:45:1);
\draw[thick, smooth] ({2-sqrt(2)/2},{1+sqrt(2)/2})--({2+sqrt(2)/2},{1+3*sqrt(2)/2});
\draw[thick, smooth] ({2+sqrt(2)/2},{1+sqrt(2)/2})--({2-sqrt(2)/2},{1+3*sqrt(2)/2});
\draw[thick, smooth] (3,{1 + 2*sqrt(2)}) arc (0:-45:1);
\draw[thick, smooth] (1,{1 + 2*sqrt(2)}) arc (-180:-135:1);
\draw[thick, smooth] (3,{1 + 2*sqrt(2)}) arc (0:45:1);
\draw[thick, smooth] (1,{1 + 2*sqrt(2)}) arc (180:135:1);
\draw[thick, smooth] ({2-sqrt(2)/2},{1+5*sqrt(2)/2})--({2+sqrt(2)/2},{1+7*sqrt(2)/2});
\draw[thick, smooth] ({2+sqrt(2)/2},{1+5*sqrt(2)/2})--({2-sqrt(2)/2},{1+7*sqrt(2)/2});
\draw[thick, smooth] (3,{1 + 4*sqrt(2)}) arc (0:-45:1);
\draw[thick, smooth] (1,{1 + 4*sqrt(2)}) arc (-180:-135:1);
\draw[thick,smooth] (1,{1 + 4*sqrt(2)})--(1,{1 + 4*sqrt(2) + 1}); 
\draw[thick,smooth] (3,{1 + 4*sqrt(2)})--(3,{1 + 4*sqrt(2) + 1});
\node[label={above: \fs ${\color{red} \begin{array}{c}  \shortuparrow \\ {\color{black} \b{x}} \end{array} }$}] at (1,{1 + 3*sqrt(2) + 2}) {};
\node[label={above: \fs ${\color{red} \begin{array}{c}  \shortuparrow \\ {\color{black} \b{y}} \end{array} }$}] at (3,{1 + 3*sqrt(2) + 2}) {};

\node at (3,1) {$\bullet$};
\node at (3,{1+2*sqrt(2)}) {$\bullet$};

\node at (5,{(1+2*sqrt(2))}) {$=$};

\node[label={below: \fs ${\color{red} \begin{array}{c} {\color{black} x} \\ \shortuparrow \end{array} }$}] at (7,0.5) {};
\node[label={below: \fs ${\color{red} \begin{array}{c} {\color{black} y} \\ \shortuparrow \end{array} }$}] at (9,0.5) {};

\draw[thick,smooth] (7,0)--(7,1); 
\draw[thick,smooth] (9,0)--(9,1);
\draw[thick,smooth] (7,1) arc (180:135:1);
\draw[thick,smooth] (9,1) arc (0:45:1);
\draw[thick, smooth] ({8-sqrt(2)/2},{1+sqrt(2)/2})--({8+sqrt(2)/2},{1+3*sqrt(2)/2});
\draw[thick, smooth] ({8+sqrt(2)/2},{1+sqrt(2)/2})--({8-sqrt(2)/2},{1+3*sqrt(2)/2});
\draw[thick, smooth] (9,{1 + 2*sqrt(2)}) arc (0:-45:1);
\draw[thick, smooth] (7,{1 + 2*sqrt(2)}) arc (-180:-135:1);
\draw[thick, smooth] (9,{1 + 2*sqrt(2)}) arc (0:45:1);
\draw[thick, smooth] (7,{1 + 2*sqrt(2)}) arc (180:135:1);
\draw[thick, smooth] ({8-sqrt(2)/2},{1+5*sqrt(2)/2})--({8+sqrt(2)/2},{1+7*sqrt(2)/2});
\draw[thick, smooth] ({8+sqrt(2)/2},{1+5*sqrt(2)/2})--({8-sqrt(2)/2},{1+7*sqrt(2)/2});
\draw[thick, smooth] (9,{1 + 4*sqrt(2)}) arc (0:-45:1);
\draw[thick, smooth] (7,{1 + 4*sqrt(2)}) arc (-180:-135:1);
\draw[thick,smooth] (7,{1 + 4*sqrt(2)})--(7,{1 + 4*sqrt(2) + 1}); 
\draw[thick,smooth] (9,{1 + 4*sqrt(2)})--(9,{1 + 4*sqrt(2) + 1});

\node[label={above: \fs ${\color{red} \begin{array}{c}  \shortuparrow \\ {\color{black} \b{x}} \end{array} }$}] at (7,{1 + 3*sqrt(2) + 2}) {};
\node[label={above: \fs ${\color{red} \begin{array}{c}  \shortuparrow \\ {\color{black} \b{y}} \end{array} }$}] at (9,{1 + 3*sqrt(2) + 2}) {};

\node at (9,{1+2*sqrt(2)}) {$\bullet$};
\node at (9,{1+4*sqrt(2)}) {$\bullet$};

\end{tikzpicture}
\caption{Sklyanin's reflection equation. As with the regular Yang--Baxter equation, a fixed and common choice is made for the arrows on the four external edges on both sides of this equation. The internal edges are summed over. Hence this gives rise to $2^4$ identities involving the Boltzmann weights. Notice that the rapidity variable on each line is reciprocated after the line passes through the dot $\bullet$ situated at the boundary.}
\label{fig:refl}
\end{figure}
The partition function relevant to this section is one that was first considered by Tsuchiya \cite{tsu}, namely the domain wall partition function with a reflecting boundary, see Figure \ref{fig:refl-dwpf}. When used in conjunction with the regular Yang--Baxter equation, the reflection equation allows one to deduce that this partition function is symmetric in the set $\{x_1,\dots,x_n\}$ (while symmetry in $\{y_1,\dots,y_n\}$ just follows from the Yang--Baxter equation itself).

\begin{figure}[H]
\begin{tikzpicture}[scale=0.6]
\foreach\x in {1,...,8}
\draw[thick]
(0,\x) -- (4.5,\x);

\foreach\x in {1,...,4}
\node[label={left: \fs ${\color{red} \shortrightarrow} \ x_{\x}$}] at (0,10-2*\x-1) {};
\foreach\x in {1,...,4}
\node[label={left: \fs ${\color{red} \shortleftarrow} \ \b{x}_{\x}$}] at (0,10-2*\x) {};

\foreach\y in {1,...,8}
\node at (0.5,\y) {\r};

\foreach\x in {1,...,4}
\draw[thick]
(\x,0) -- (\x,9);
\foreach\x in {1,...,4}
\node[label={below: \fs ${\color{red} \begin{array}{c} {\color{black} \b{y}_{\x}} \\ \shortuparrow \end{array} }$}] at (5-\x,0) {};

\foreach\x in {1,...,4}{
\node at (\x,0.5) {\d};
\node at (\x,8.5) {\u};
}

\foreach\x in {1,...,4}
\draw[thick, smooth] (4.5,2*\x-1) arc (-90:90:0.5);
\foreach\x in {1,...,4}
\node at (5,2*\x-1/2) {$\bullet$};

\end{tikzpicture}
\caption{$Z_{\rm UASM}$, the six-vertex model partition function with reflecting domain wall boundary conditions in the case $n=4$. We draw attention to the fact that every second horizontal line has right-to-left orientation. This must be taken into consideration when evaluating the Boltzmann weights in these rows, following the convention explained in the caption of Figure \ref{6vertices}.}
\label{fig:refl-dwpf}
\end{figure}
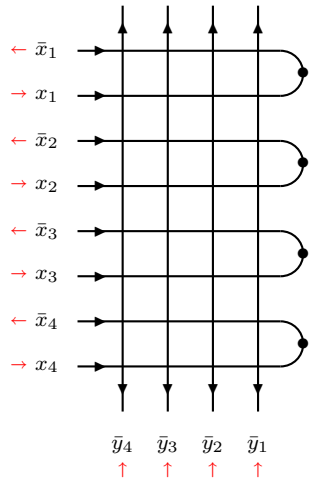
Tsuchiya was able to evaluate this partition function in determinant form \cite{tsu}:
\begin{multline}
\label{tsu-det}
Z_{\rm UASM}(x_1,\dots,x_n; y_1,\b{y}_1,\dots,y_n,\b{y}_n; t)
=
\\
\frac{
\prod_{i,j=1}^{n} 
(1-t x_i y_j) 
(1-t x_i \b{y}_j)
}
{\prod_{1\leq i<j \leq n} (x_i-x_j) (y_i-y_j) (1 - t x_i x_j) (1 - \b{y}_i \b{y}_j) }
\det\left[
\frac{(1-t)}{(1- x_i y_j)(1-t x_i y_j)(1- x_i \b{y}_j)(1-t x_i \b{y}_j)}
\right]_{1\leq i,j \leq n}.
\end{multline}
As in the case of the DWPF, one can prove this determinant expression by writing down a set of conditions which uniquely determine $Z_{\rm UASM}$, and showing that \eref{tsu-det} satisfies these conditions.

As our notation suggests, $Z_{\rm UASM}$ can be considered as a multiparameter generating series of U-turn ASMs \cite{kup2}. At the same time, up to some overall product terms which differ, $Z_{\rm UASM}$ appears on the right hand side of \eref{symp-pp-UASM}. Hence \eref{symp-pp-UASM} relates a generating series of symplectic plane partitions with a generating series of UASMs (up to proportionality). We would like to make a similar statement regarding \eref{uasm-conj}, since $Z_{\rm UASM}$ appears {\it identically} on the right hand side, but are unable to do so since we have no combinatorial interpretation of the left hand side.

\subsection{Rectangular domains with reflecting boundaries}
\label{ssec:p-uasm}

Here we perform similar analysis to that of Section \ref{ssec:p-dwpf}, and consider the case where the cardinalities of the two sets of variables in \eref{uasm-conj} are not equal. The specialization of interest is once again $x_{m+1} = \cdots = x_n = 0$, for some $m<n$. Calculating the left hand side of \eref{uasm-conj} in this limit is trivial, while the right hand side can be computed by manipulating the Tsuchiya determinant into the following form:    
\begin{align*}
\det\left[
\frac{(1-t)}{(1- x_i y_j)(1-t x_i y_j)(1- x_i \b{y}_j)(1-t x_i \b{y}_j)}
\right]_{1\leq i,j \leq n}
=
\frac{
\prod_{m+1 \leq i<j \leq n}
(x_i-x_j)
}
{
\prod_{i=m+1}^{n}
(1-t x_i^2)
}
\det\Big[
\mathcal{U}_{i,j}
\Big]_{1 \leq i,j \leq n},
\end{align*}
where the entries of the final determinant are given by
\begin{align*}
\mathcal{U}_{i,j}
=
\left\{
\begin{array}{ll}
\mathcal{U}_{i,j}(x_i;y_j)
=
\displaystyle{
\frac{(1-t)}{(1-x_i y_j)(1-t x_i y_j)(1-x_i \b{y}_j)(1-t x_i \b{y}_j)}
},
\quad
&  
1 \leq i \leq m,
\\ \\ 
\mathcal{U}_{i,j}(x_i,\dots,x_n;y_j)
=
\displaystyle{\sum_{k=0}^{\infty}}
(1-t^{k+1})h_{k+i-n}(x_i,\dots,x_n)
(y_j^{k+1}-\b{y}_j^{k+1})/(y_j-\b{y}_j),
\quad
& 
m+1 \leq i \leq n.
\end{array}
\right.
\end{align*}
After this rearrangement of the determinant in \eref{uasm-conj}, $x_i \rightarrow 0$ for all $m+1 \leq i \leq n$ is no longer a singular limit. We therefore obtain:
\newline

\noindent \textbf{Conjecture 1$\bm{'}$.}
\begin{multline}
\label{uasm-conj-pdwpf}
\sum_{\lambda}
\prod_{j=1}^{n-\ell(\lambda)}
(1-t^j)
b_{\lambda}(t)
P_{\lambda}(x_1,\dots,x_m;t)
K_{\lambda}(y_1,\b{y}_1,\dots,y_n,\b{y}_n;t)
=
\\
\frac{
\prod_{i=1}^{n-m} (1-t^i)
\prod_{i=1}^{m} \prod_{j=1}^{n} 
(1-t x_i y_j) 
(1-t x_i \b{y}_j)
}
{\prod_{i=1}^{m} (x_i^{n-m}) \prod_{1\leq i<j \leq m} (x_i-x_j) (1 - t x_i x_j) 
\prod_{1\leq i<j \leq n} (y_i-y_j) (1 - \b{y}_i \b{y}_j)}
\det\Big[
\mathcal{U}^{\circ}_{i,j}
\Big]_{1 \leq i,j \leq n},
\end{multline} 
\textit{where the entries of the determinant are given by}
\begin{align*}
\mathcal{U}^{\circ}_{i,j}
=
\frac{(1-t)}{(1-x_i y_j)(1-t x_i y_j)(1-x_i \b{y}_j)(1-t x_i \b{y}_j)},
\ \  
1 \leq i \leq m,
\quad
\mathcal{U}^{\circ}_{i,j}
=
\frac{(y_j^{n-i+1}-\b{y}_j^{n-i+1})}{(y_j-\b{y}_j)},
\ \ 
m+1 \leq i \leq n.
\end{align*}

As we did in Section \ref{ssec:p-dwpf}, we consider the meaning of the right hand side of \eref{uasm-conj-pdwpf} in six-vertex model terms. Due to the trivial form of the Boltzmann weights \eref{triv-boltz} in this limit, no $c_{-}$ vertices are allowed in the bottom $n-m$ double rows of Figure \ref{fig:refl-dwpf}. This forces the U-turn vertex at the end of each of these double rows to be in the $k_{-}(x_i)$ configuration, since otherwise their top row would necessarily contain a $c_{-}$ vertex. After making this observation, the methods of \cite{fw2} can be applied in much the same way to calculate the required limit. Once again, a simple calculation shows that these double rows contribute only the multiplicative factor $\prod_{i=1}^{n-m} (1-t^i)$ to the partition function, and they can thus be deleted from the lattice. The bottom edges of the resulting lattice are summed over all arrow configurations, as illustrated in Figure \ref{fig:p-refl-dwpf}. We denote this partition function by $Z_{\rm UASM}(x_1,\dots,x_m; y_1,\b{y}_1,\dots,y_n,\b{y}_n;t)$. Cancelling the factor of $\prod_{i=1}^{n-m} (1-t^i)$ from the right hand side of \eref{uasm-conj-pdwpf}, it is equal to 
$Z_{\rm UASM}(x_1,\dots,x_m; y_1,\b{y}_1,\dots,y_n,\b{y}_n;t)$.

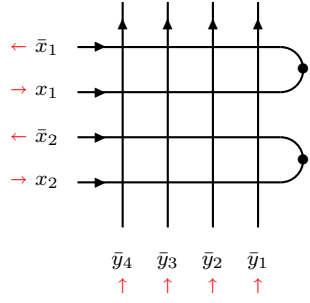
\begin{figure}[H]
\begin{tikzpicture}[scale=0.6]
\foreach\x in {1,...,4}
\draw[thick]
(0,\x) -- (4.5,\x);

\foreach\x in {1,...,2}
\node[label={left: \fs ${\color{red} \shortrightarrow} \ x_{\x}$}] at (0,6-2*\x-1) {};
\foreach\x in {1,...,2}
\node[label={left: \fs ${\color{red} \shortleftarrow} \ \b{x}_{\x}$}] at (0,6-2*\x) {};

\foreach\y in {1,...,4}
\node at (0.5,\y) {\r};

\foreach\x in {1,...,4}
\draw[thick]
(\x,0) -- (\x,5);
\foreach\x in {1,...,4}
\node[label={below: \fs ${\color{red} \begin{array}{c} {\color{black} \b{y}_{\x}} \\ \shortuparrow \end{array} }$}] at (5-\x,0) {};

\foreach\x in {1,...,4}
\node at (\x,4.5) {\u};

\foreach\x in {1,...,2}
\draw[thick, smooth] (4.5,2*\x-1) arc (-90:90:0.5);
\foreach\x in {1,...,2}
\node at (5,2*\x-1/2) {$\bullet$};

\end{tikzpicture}
\caption{$Z_{\rm UASM}(x_1,\dots,x_m; y_1,\b{y}_1,\dots,y_n,\b{y}_n;t)$ in the case $m=2$ and $n=4$, which results from setting $x_{m+1} = \cdots = x_n = 0$ in Figure \ref{fig:refl-dwpf}. The lower external edges are left blank, to indicate summation over all possible arrow configurations.}
\label{fig:p-refl-dwpf}
\end{figure}



\section{Refined Littlewood identities and off-diagonally symmetric alternating sign matrices}
\label{sec:OSASM}

\subsection{A Pfaffian formula} Let $x_1,\dots,x_{2n}$ be indeterminates. The following Pfaffian formula is due to Laksov, Lascoux and Thorup \cite{llt} and Stembridge \cite{ste}:
\begin{align}
{\rm Pf}\left[ \frac{x_i-x_j}{1-x_i x_j} \right]_{1\leq i < j \leq 2n}
=
\prod_{1 \leq i<j \leq 2n} \frac{ x_i-x_j}{1-x_i x_j}.
\label{stem-pf}
\end{align}
One can prove this by factor exhaustion.

\subsection{Littlewood identities for Schur polynomials and symmetric plane partitions}
\label{schur-littlewood}

In contrast to the Cauchy identities studied so far, which are essentially unique within any particular family of symmetric functions, Littlewood identities come in greater abundance \cite{mac}. For example, in the family of the Schur polynomials alone, the following identities are known:
\begin{align}
\sum_{\lambda} s_{\lambda}(x_1,\dots,x_n)
&=
\prod_{1 \leq i<j \leq n}
\frac{1}{1-x_i x_j}\ 
\prod_{i=1}^{n}
\frac{1}{1-x_i},
\label{s-little1}
\\
\sum_{\lambda\ \text{even}} s_{\lambda}(x_1,\dots,x_n)
&=
\prod_{1 \leq i < j \leq n}
\frac{1}{1-x_i x_j}\ 
\prod_{i=1}^{n}
\frac{1}{1-x_i^2},
\label{s-little2}
\\
\sum_{\lambda'\ \text{even}} s_{\lambda}(x_1,\dots,x_n)
&=
\prod_{1 \leq i<j \leq n}
\frac{1}{1-x_i x_j}.
\label{s-little3}
\end{align}
It is possible to regard the left hand side of each of these identities as a generating series of {\it symmetric plane partitions}, \ie plane partitions whose entries satisfy $\pi(i,j) = \pi(j,i)$ for all $i,j$. The only subtlety is a possible restriction imposed on the central slice of these plane partitions, which varies according to the particular identity. Let us define the set of symmetric plane partitions
\begin{align*}
\bm{\pi}^{\text{s}}_n
=
\{\emptyset \equiv \lambda^{(0)} \prec \lambda^{(1)} \prec \cdots \prec \lambda^{(n)}
\succ \cdots \succ \lambda^{(1)} \succ \lambda^{(0)} \equiv \emptyset \} 
\end{align*}
and consider the left hand side of equations \eref{s-little1}--\eref{s-little3} as a sum over the elements in $\bm{\pi}^{\text{s}}_n$:

\begin{itemize} 

\item In the case of \eref{s-little1}, the sum is over all symmetric plane partitions, with no restriction on the central slice:
\begin{align*}
\sum_{\pi \in \bm{\pi}^{\text{s}}_n}
\
\prod_{i=1}^{n} x_i^{|\lambda^{(i)}|-|\lambda^{(i-1)}|}
=
\prod_{1 \leq i<j \leq n}
\frac{1}{1-x_i x_j}\ 
\prod_{i=1}^{n}
\frac{1}{1-x_i}.
\end{align*}
As with ordinary (unsymmetric) plane partitions, one can consider $q$-specializations of the variables. Setting $x_i = q^{2(n-i)+1}$, it is possible to weight the plane partitions by their volume:
\begin{align*}
\sum_{\pi \in \bm{\pi}^{\text{s}}_n}
q^{|\pi|}
=
\prod_{1 \leq i<j \leq n}
\frac{1}{1-q^{2(i+j-1)}}
\prod_{i=1}^{n}
\frac{1}{1-q^{2i-1}}.
\end{align*}

\item In the case of \eref{s-little2}, one sums over symmetric plane partitions whose central slice is an {\it even partition}, \ie a partition with only even parts:
\begin{align*}
\sum_{\substack{ \pi \in \bm{\pi}^{\text{s}}_n \\ \pi(k,k)\ \text{even} }}
\
\prod_{i=1}^{n} x_i^{|\lambda^{(i)}|-|\lambda^{(i-1)}|}
=
\prod_{1 \leq i < j \leq n}
\frac{1}{1-x_i x_j}\ 
\prod_{i=1}^{n}
\frac{1}{1-x_i^2}.
\end{align*}

\item In the final case \eref{s-little3}, the sum is over symmetric plane partitions subject to the condition that all connected components crossing the central slice must be of even width:
\begin{align}
\label{s-little3-pp-gs}
\sum_{\substack{ \pi \in \bm{\pi}^{\text{s}}_n \\ \pi(2k-1,2k-1) = \pi(2k,2k) }}
\
\prod_{i=1}^{n} x_i^{|\lambda^{(i)}|-|\lambda^{(i-1)}|}
=
\prod_{1 \leq i<j \leq n}
\frac{1}{1-x_i x_j}.
\end{align}
See Figure \ref{fig:sym-pp} for an example of a symmetric plane partition which satisfies this criterion. Of the three Littlewood identities listed, \eref{s-little3} will turn out to be most relevant for our purposes.

\end{itemize}
  
\begin{figure}
\includegraphics[scale=0.4]{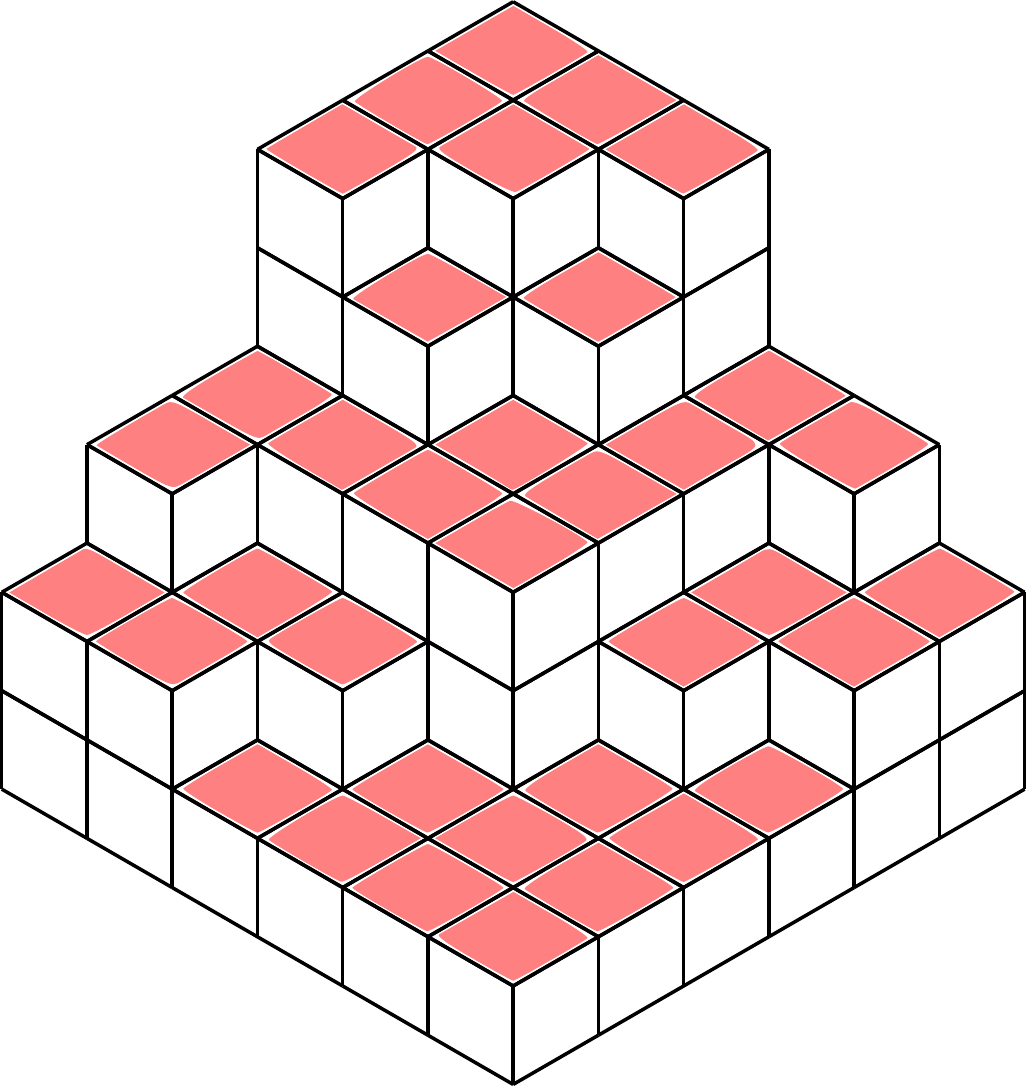}
\caption{A symmetric plane partition, with central diagonal slice $\lambda = (5,5,3,3,1,1)$ such that 
$\lambda' = (6,4,4,2,2)$ is even. Equivalently, all connected components crossing the central slice must have even width there.}
\label{fig:sym-pp}
\end{figure}

\subsection{Littlewood identities for Hall--Littlewood polynomials and associated plane partitions}

The three Littlewood identities presented in Section \ref{schur-littlewood} admit $t$-generalizations to Hall--Littlewood polynomials (see \cite{mac}, Chapter III, Section 5, Example 3):
\begin{align}
\label{HL-little1}
\sum_{\lambda}
P_{\lambda} (x_1,\dots,x_n;t)
&=
\prod_{1\leq i<j \leq n}
\frac{1-t x_i x_j}{1-x_i x_j}\ 
\prod_{i=1}^{n}
\frac{1}{1-x_i},
\\
\label{HL-little2}
\sum_{
\lambda\ \text{even}
}
P_{\lambda} (x_1,\dots,x_n;t)
&=
\prod_{1\leq i<j \leq n}
\frac{1-t x_i x_j}{1-x_i x_j}\
\prod_{i=1}^{n}
\frac{1}{1-x_i^2},
\\
\label{HL-little3}
\sum_{
\lambda'\ \text{even}
}
\ \
\prod_{i=1}^{\infty}\ \prod_{j=2,4,6,\dots}^{m_i(\lambda)}
(1-t^{j-1})
P_{\lambda}(x_1,\dots,x_n;t)
&=
\prod_{1\leq i<j \leq n}
\frac{1-t x_i x_j}{1-x_i x_j}.
\end{align}
All of these identities can be viewed as generating series of symmetric plane partitions, with an appropriate $t$-weighting. To specify these generating series precisely, we need to define two further statistics related to paths on plane partitions.
\begin{defn}[Paths (not) crossing the main diagonal]
Let $\pi$ be a symmetric plane partition. Let $2p^{\circ}_d(\pi)$ denote the number of paths in $\pi$ at depth $d$, which stay entirely within one half of $\pi$ and do not intersect with the main diagonal. Similarly, we let $p^{\bullet}_d(\pi)$ denote the number of paths in $\pi$ at depth $d$, which do intersect with the main diagonal.
\end{defn}
Armed with these definitions, we study the left hand side of the identities 
\eref{HL-little1}--\eref{HL-little3} from a combinatorial point of view: 

\begin{itemize}

\item In \eref{HL-little1}, each path that crosses the main diagonal {\it does not} receive a $t$-weighting. All other paths come in pairs (due to symmetry). Each pair of paths of depth $k$ receives a weight of $1-t^k$:
\begin{align*}
\sum_{\pi \in \bm{\pi}^{\text{s}}_n}
\
\prod_{i \geq 1}
(1-t^i)^{p^{\circ}_i(\pi)}
\prod_{i=1}^{n} x_i^{|\lambda^{(i)}|-|\lambda^{(i-1)}|}
=
\prod_{1\leq i<j \leq n}
\frac{1-t x_i x_j}{1-x_i x_j}\ 
\prod_{i=1}^{n}
\frac{1}{1-x_i}.
\end{align*}

\item The $t$-weighting in \eref{HL-little2} is the same as in \eref{HL-little1}, but the sum is taken over symmetric plane partitions with an even central slice:
\begin{align*}
\sum_{\substack{ \pi \in \bm{\pi}^{\text{s}}_n \\ \pi(k,k)\ \text{even} }}
\
\prod_{i \geq 1}
(1-t^i)^{p^{\circ}_i(\pi)}
\prod_{i=1}^{n} x_i^{|\lambda^{(i)}|-|\lambda^{(i-1)}|}
=
\prod_{1\leq i<j \leq n}
\frac{1-t x_i x_j}{1-x_i x_j}\
\prod_{i=1}^{n}
\frac{1}{1-x_i^2}.
\end{align*}
 
\item In \eref{HL-little3}, each path crossing the main diagonal at depth $k$, where $k$ is odd, receives a weight of $1-t^k$. Paths crossing the main diagonal with even depth do not receive a $t$-weighting. Each pair of paths (away from the main diagonal) of depth $k$ receives a weight of $1-t^k$, as usual:
\begin{align}
\label{hl-little3-pp-gs}
\sum_{\substack{ \pi \in \bm{\pi}^{\text{s}}_n \\ \pi(2k-1,2k-1) = \pi(2k,2k) }}
\
\prod_{i \geq 1}
(1-t^i)^{p^{\circ}_i(\pi)}
\prod_{j=1,3,5,\dots}
(1-t^j)^{p^{\bullet}_j(\pi)}
\prod_{i=1}^{n} x_i^{|\lambda^{(i)}|-|\lambda^{(i-1)}|}
=
\prod_{1 \leq i<j \leq n}
\frac{1-t x_i x_j}{1-x_i x_j}.
\end{align}
    
\end{itemize}

\subsection{Refined Littlewood identities}

Again analogously to Theorems \ref{thm1} and \ref{thm2}, we state a theorem (for Schur polynomials) and a conjecture (for Hall--Littlewood polynomials) which deal with the expansion of a certain Pfaffian (in this case, the partition function of OSASMs under certain weights) in those polynomials. Both are $t$-refinements of the Littlewood identity \eref{s-little3} (which can be obtained by setting $t=0$, and using the Pfaffian factorization \eref{stem-pf}). The conjecture is also a natural deformation of the Littlewood identity \eref{HL-little3} for Hall--Littlewood polynomials, while not being a true refinement thereof.

\begin{thm}
\label{thm4}
\begin{align}
\label{s-refined-little}
\sum_{\lambda'\ {\rm even}}\ 
\prod_{i=2,4,6,\dots}^{2n}
(1-t^{\lambda_i-i+2n+1})
s_{\lambda}(x_1,\dots,x_{2n})
=
\prod_{1\leq i<j \leq 2n}
\frac{1}{(x_i-x_j)}
{\rm Pf}
\left[
\frac{(x_i - x_j)(1-t)}
{
(1 - x_i x_j)
(1 - t x_i x_j) 
}
\right]_{1\leq i < j \leq 2n}.
\end{align}
 
\end{thm}

\begin{proof}
The proof proceeds along similar lines to the proof of Theorem \ref{thm1}. We treat the entries of the Pfaffian as formal power series:
\begin{align*}
{\rm Pf}
\left[
\frac{(x_i - x_j)(1-t)}
{
(1 - x_i x_j)
(1 - t x_i x_j) 
}
\right]_{1 \leq i<j \leq 2n}
&=
{\rm Pf}
\left[
(x_i - x_j)
\sum_{k=0}^{\infty}
(1-t^{k+1})
x_i^k x_j^k
\right]_{1 \leq i<j \leq 2n}
\\
&=
{\rm Pf}
\left[
\sum_{0\leq k < l}
\delta_{l,k+1} (1-t^{k+1}) 
(x_i^l x_j^k - x_i^k x_j^l)
\right]_{1 \leq i<j \leq 2n}.
\end{align*}
From here, we use the Pfaffian analogue of the Cauchy--Binet identity \eref{cb-analog2} given in Appendix A, with $m = 2n$, $M \rightarrow \infty$ and $A_{kl} = -\delta_{l,k+1} (1-t^k)$. We conclude that
\begin{align*}
{\rm Pf}
\left[
\frac{(x_i - x_j)(1-t)}
{
(1 - x_i x_j)
(1 - t x_i x_j) 
}
\right]_{1 \leq i<j \leq 2n}
=
\sum_{1 \leq s_1 < \cdots < s_{2n}}
(-1)^{n}
\prod_{i=2,4,6,\dots}^{2n}
\Big\{
(1-t^{s_{i-1}})\ \delta_{s_{i-1},s_i-1}
\Big\}
\det\left[
x_i^{s_j-1}
\right]_{1\leq i,j \leq 2n}
\end{align*}
where we have used the factorization of the Pfaffian
\begin{align*}
{\rm Pf}
\Big[
\delta_{s_i+1,s_j}
(1-t^{s_i})
\Big]_{1 \leq i<j \leq 2n} 
= 
\prod_{i=2,4,6,\dots}^{2n} 
(1-t^{s_{i-1}})\ \delta_{s_{i-1},s_i-1}.
\end{align*}
Making the change of summation indices $s_{2n-i+1} = \lambda_i - i + 2n+1$, this becomes
\begin{align*}
{\rm Pf}
\left[
\frac{(x_i - x_j)(1-t)}
{
(1 - x_i x_j)
(1 - t x_i x_j) 
}
\right]_{1 \leq i<j \leq 2n}
&=
\sum_{\substack{ 
\lambda_1 \geq \cdots \geq \lambda_{2n} \geq 0 \\ \lambda_{2k-1} = \lambda_{2k} 
}}
\
(-1)^n
\prod_{i=2,4,6,\dots}^{2n}
(1-t^{\lambda_i-i+2n+1})
\det\left[
x_i^{\lambda_{2n-j+1}+j-1}
\right]_{1\leq i,j \leq 2n}
\\
&=
\sum_{\lambda'\ \text{even}}
\
\prod_{i=2,4,6,\dots}^{2n}
(1-t^{\lambda_i-i+2n+1})
\det\left[
x_i^{\lambda_j-j+2n}
\right]_{1\leq i,j \leq 2n}
\end{align*}
and the proof is complete after dividing both sides by the Vandermonde $\Delta(x)_{2n}$.

\end{proof}

The next conjecture (which again has been checked in Mathematica for partitions of small size) we plan to address and prove in a subsequent paper with P.~Zinn-Justin. One would hope a proof which involved some appropriate difference operators existed, but we have not been able to construct said operators. 

\begin{conj}
\label{conj2}
\begin{multline}
\label{osasm-conj}
\sum_{
\lambda'\ {\rm even}
}
\ \
\prod_{i=0}^{\infty}\ \prod_{j=2,4,6,\dots}^{m_i(\lambda)}
(1-t^{j-1})
P_{\lambda}(x_1,\dots,x_{2n};t)
\\
=
\prod_{1 \leq i<j \leq 2n}
\frac{(1-t x_i x_j)}{(x_i-x_j)}
{\rm Pf}
\left[
\frac{(x_i-x_j)(1-t)}{(1-x_i x_j) (1-t x_i x_j)}
\right]_{1\leq i < j \leq 2n}.
\end{multline}
\end{conj}

\subsection{Refined symmetric plane partitions}

As with the refined Cauchy identities in Section \ref{sec:refined-cauchy}, one can consider the interpretation of the refined Littlewood identities \eref{s-refined-little} and \eref{osasm-conj} at the level of plane partitions. The refinements that these identities produce are quite analogous to those seen in Section \ref{sec:refined-pp}. The first case \eref{s-refined-little} can be viewed as a $t$-refinement of the generating series \eref{s-little3-pp-gs}, which assigns a $t$-dependent weight to the central slice of the symmetric plane partition:
\begin{multline*}
\sum_{\substack{ \pi \in \bm{\pi}^{\text{s}}_{2n} \\ \pi(2k-1,2k-1) = \pi(2k,2k) }}
\
\prod_{i=2,4,6,\dots}^{2n}
(1-t^{\pi(i,i)-i+2n+1})
\prod_{j=1}^{2n} x_j^{|\lambda^{(j)}|-|\lambda^{(j-1)}|}
=
\\
\prod_{1\leq i<j \leq 2n}
\frac{1}{(x_i-x_j)}
{\rm Pf}
\left[
\frac{(x_i - x_j)(1-t)}
{
(1 - x_i x_j)
(1 - t x_i x_j) 
}
\right]_{1\leq i < j \leq 2n}.
\end{multline*}
The second case \eref{osasm-conj} is a natural modification of the generating series \eref{hl-little3-pp-gs}, which now takes into consideration paths at height 0.
\begin{defn}
Let $\pi$ be a symmetric plane partition whose base is contained within a $2n \times 2n$ square. We let $\tilde{p}^{\bullet}_d(\pi)_{2n \times 2n}$ denote the number of paths in $\pi$ with height $h \geq 0$ and depth $d$, and which intersect with the main diagonal. The definition of a path at height 0 is the same as that given in Definition \ref{defn:height0}.
\end{defn}
The conjecture \eref{osasm-conj} translates to the modified generating series
\begin{multline}
\label{sym-pp-OSASM}
\sum_{\substack{ \pi \in \bm{\pi}^{\text{s}}_{2n} \\ \pi(2k-1,2k-1) = \pi(2k,2k) }}
\
\prod_{i \geq 1}
(1-t^i)^{p^{\circ}_i(\pi)}
\prod_{j=1,3,5,\dots}
(1-t^j)^{\tilde{p}^{\bullet}_j(\pi)_{2n \times 2n}}
\prod_{i=1}^{2n} x_i^{|\lambda^{(i)}|-|\lambda^{(i-1)}|}
=
\\
\prod_{1 \leq i<j \leq 2n}
\frac{(1-t x_i x_j)}{(x_i-x_j)}
{\rm Pf}
\left[
\frac{(x_i-x_j)(1-t)}{(1-x_i x_j) (1-t x_i x_j)}
\right]_{1\leq i < j \leq 2n}.
\end{multline}
For an illustration of this modification on a typical symmetric plane partition, see Figure 
\ref{fig:modified-symm-pp}.

\begin{figure}
\begin{tabular}{ccc}
\includegraphics[scale=0.4]{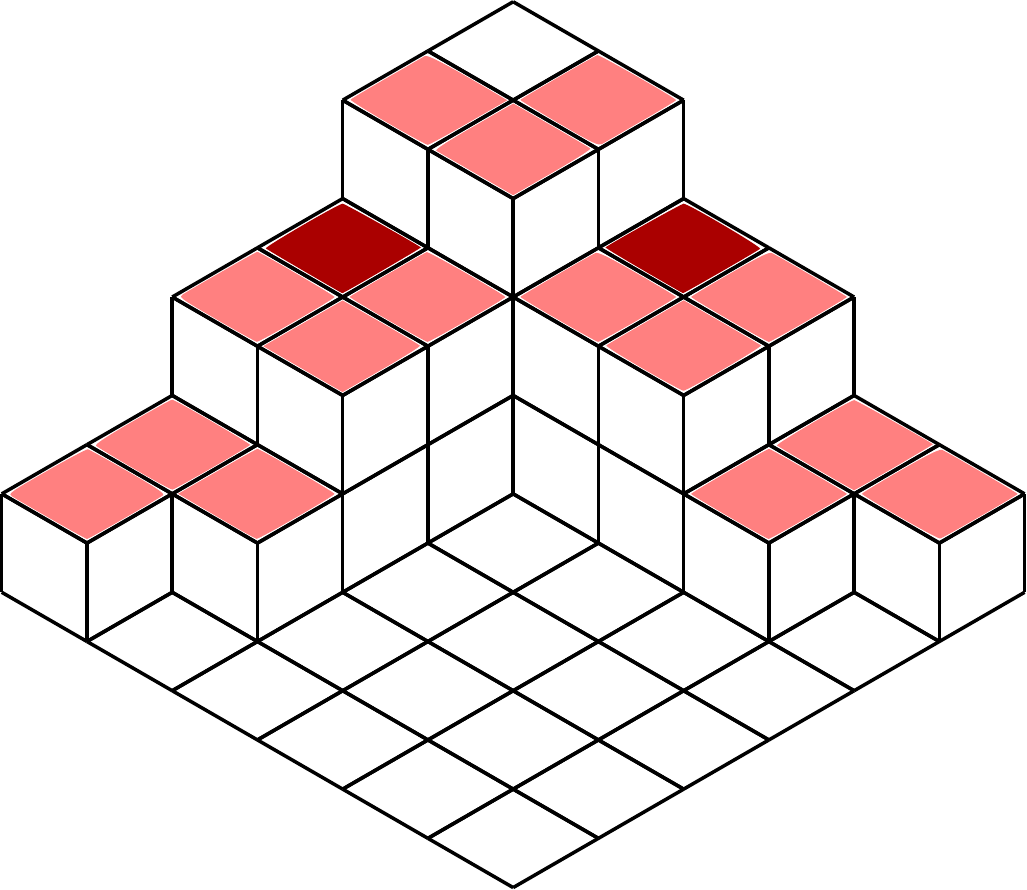}
& &
\includegraphics[scale=0.4]{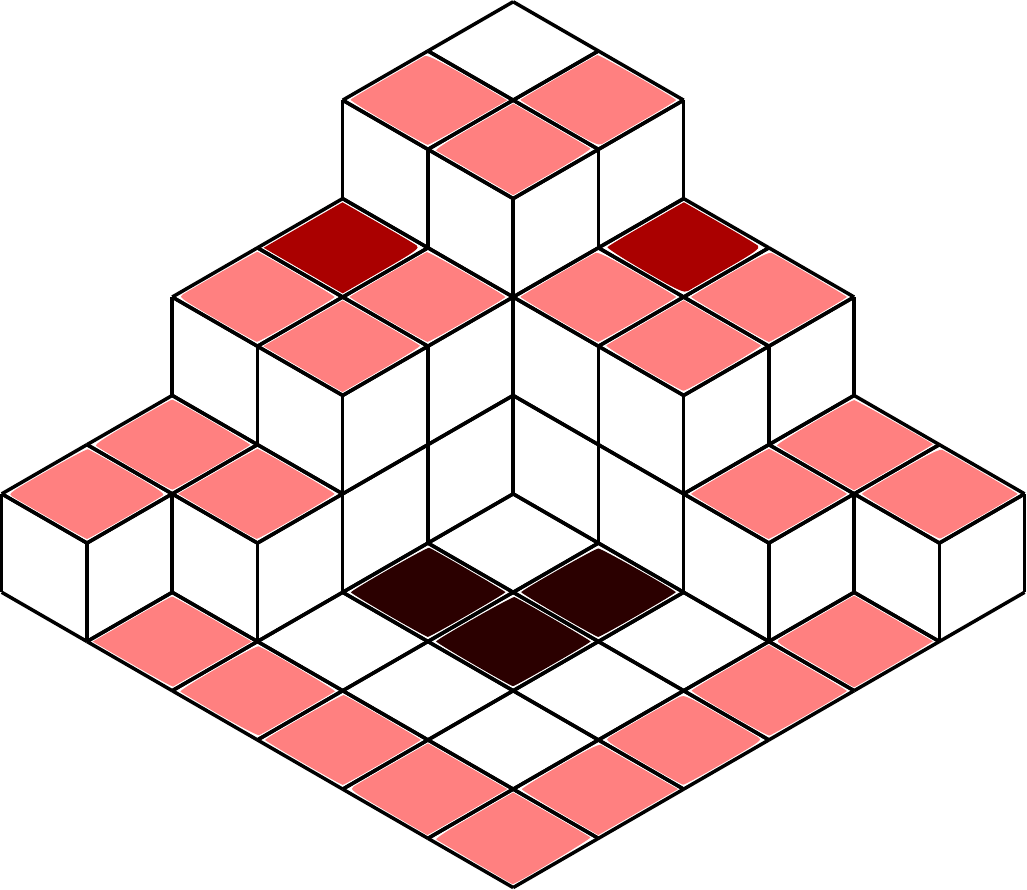}
\end{tabular}
\caption{On the left, a path-weighted symmetric plane partition $\pi$. Using equation \eref{hl-little3-pp-gs}, it receives a $t$-weighting of $\prod_{i \geq 1} (1-t^i)^{p_i^{\circ}(\pi)} \prod_{j=1,3,5,\dots} (1-t^j)^{p_j^{\bullet}(\pi)} = (1-t)^3 (1-t^2)$. On the right, the additional height-0 paths which must be taken into consideration when studying the modified Littlewood identity \eref{sym-pp-OSASM}. The new weighting is $(1-t)^4 (1-t^2) (1-t^3)$.}
\label{fig:modified-symm-pp}
\end{figure}

\subsection{Six-vertex model on OSASM lattice}

In \cite{kup2}, Kuperberg introduced several new types of partition functions within the framework of the six-vertex model. Among these was the {\it off-diagonally symmetric} partition function, which can be viewed as a DWPF whose configurations are constrained to be symmetric about a central diagonal axis, and to have no $c_{\pm}$ vertices along that axis. The configurations of this partition function are in one-to-one correspondence with so-called off-diagonally symmetric ASMs (OSASMs). In order to properly specify this partition function, one needs to introduce a further pair of vertices, these being the {\it corner vertices} shown in Figure \ref{fig:corner}.

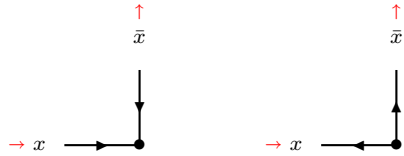
\begin{figure}[H]
\begin{tabular}{cc}
\begin{tikzpicture}[>=stealth]
\draw[thick]
(0,0) to (1,0) to (1,1);
\node at (1,0) {$\bullet$};
\node at (0.5,0) {\r};
\node at (1,0.5) {\d};
\node[label={left: \fs ${\color{red} \shortrightarrow} \ x$}] at (0,0) {};
\node[label={above: \fs ${\color{red} \begin{array}{c}  \shortuparrow \\ {\color{black} \b{x}} \end{array} }$}] at (1,1) {};
\end{tikzpicture}
\quad\quad\quad&
\begin{tikzpicture}[>=stealth]
\draw[thick]
(0,0) to (1,0) to (1,1);
\node at (1,0) {$\bullet$};
\node at (0.5,0) {\l};
\node at (1,0.5) {\u};
\node[label={left: \fs ${\color{red} \shortrightarrow} \ x$}] at (0,0) {};
\node[label={above: \fs ${\color{red} \begin{array}{c}  \shortuparrow \\ {\color{black} \b{x}} \end{array} }$}] at (1,1) {};
\end{tikzpicture}
\end{tabular}
\caption{The corner vertices, which in this work are both assigned the trivial Boltzmann weight 1.}
\label{fig:corner}
\end{figure}

The corner vertices, together with the six vertices introduced in Section \ref{sec:6v}, satisfy yet another variant of the Yang--Baxter relation, shown in Figure \ref{fig:corner-reflect}. This relation is schematically very similar to the reflection equation discussed in Section \ref{sec:U-vert}.

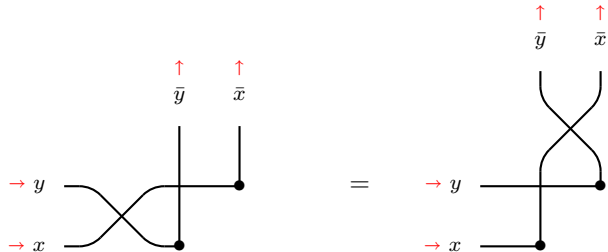
\begin{figure}[H]
\begin{tikzpicture}[scale=0.4]

\draw[thick,smooth] (-0.5,0)--(0,0);
\draw[thick,smooth] (-0.5,2)--(0,2);
\draw[thick, smooth] (0,2) arc (90:45:1);
\draw[thick, smooth] (0,0) arc (-90:-45:1);
\draw[thick, smooth] ({sqrt(2)/2},{1+sqrt(2)/2})--({3*sqrt(2)/2},{1-sqrt(2)/2});
\draw[thick, smooth] ({sqrt(2)/2},{1-sqrt(2)/2})--({3*sqrt(2)/2},{1+sqrt(2)/2});
\draw[thick, smooth] ({2*sqrt(2)},2) arc (90:135:1);
\draw[thick, smooth] ({2*sqrt(2)},0) arc (-90:-135:1);
\draw[thick, smooth] ({2*sqrt(2)}, 2)--({2*sqrt(2)+0.5}, 2);
\draw[thick, smooth] ({2*sqrt(2)}, 0)--({2*sqrt(2)+0.5}, 0);
\draw[thick, smooth] ({2*sqrt(2)+0.5}, 0)--({2*sqrt(2)+0.5}, 4);
\draw[thick, smooth] ({2*sqrt(2)+0.5+2}, 2)--({2*sqrt(2)+0.5+2}, 4);
\draw[thick, smooth] ({2*sqrt(2)+0.5}, 2)--({2*sqrt(2)+0.5+2}, 2);
\node at ({2*sqrt(2)+0.5}, 0) {$\bullet$};
\node at ({2*sqrt(2)+0.5+2}, 2) {$\bullet$};
\node[label={left: \fs ${\color{red} \shortrightarrow} \ x$}] at (-0.5, 0) {};
\node[label={left: \fs ${\color{red} \shortrightarrow} \ y$}] at (-0.5, 2) {};
\node[label={above: \fs ${\color{red} \begin{array}{c}  \shortuparrow \\ {\color{black} \b{y}} \end{array} }$}] at ({2*sqrt(2)+0.5},4) {};
\node[label={above: \fs ${\color{red} \begin{array}{c}  \shortuparrow \\ {\color{black} \b{x}} \end{array} }$}] at ({2*sqrt(2)+0.5+2},4) {};

\node at ({2*sqrt(2) + 2 + 0.5 + 4}, 2) {$=$};

\draw[thick, smooth] ({2*sqrt(2) + 2 + 0.5 + 8}, 0)--({2*sqrt(2) + 2 + 0.5 + 10}, 0);
\draw[thick, smooth] ({2*sqrt(2) + 2 + 0.5 + 8}, 2)--({2*sqrt(2) + 2 + 0.5 + 12}, 2);
\draw[thick, smooth] ({2*sqrt(2) + 2 + 0.5 + 10}, 0)--({2*sqrt(2) + 2 + 0.5 + 10}, 2);
\draw[thick, smooth] ({2*sqrt(2)+2+0.5 + 2 + 4 + 6},{2  + 2*sqrt(2) + 1})--({2*sqrt(2)+2+0.5 + 2 + 4 + 6},{2-0.5 + 2*sqrt(2) + 1});
\draw[thick, smooth] ({2*sqrt(2)+2+0.5 + 2 + 2 + 6},{2 + 2*sqrt(2) + 1})--({2*sqrt(2)+2+0.5 + 2 + 2 + 6},{2-0.5 + 2*sqrt(2) + 1});
\draw[thick, smooth] ({2*sqrt(2)+2+0.5 + 2 + 2 + 6},{2-0.5 + 2*sqrt(2) + 1}) arc (-180:-135:1);
\draw[thick, smooth] ({2*sqrt(2)+2+0.5 + 2 + 4 + 6},{2-0.5 + 2*sqrt(2) + 1}) arc (0:-45:1);
\draw[thick, smooth] ({2*sqrt(2)+2+0.5 + 2 + 2 + 1 + 6 - sqrt(2)/2},{2-0.5-sqrt(2)/2 + 2*sqrt(2) + 1})--({2*sqrt(2)+2+0.5 + 2 + 2 + 1 + 6 + sqrt(2)/2},{2-0.5-3*sqrt(2)/2 + 2*sqrt(2) + 1});
\draw[thick, smooth] ({2*sqrt(2)+2+0.5 + 2 + 2 + 1 + 6 - sqrt(2)/2},{2-0.5-3*sqrt(2)/2 + 2*sqrt(2) + 1})--({2*sqrt(2)+2+0.5 + 2 + 2 + 1 + 6 + sqrt(2)/2},{2-0.5-sqrt(2)/2 + 2*sqrt(2) + 1});
\draw[thick, smooth] ({2*sqrt(2)+2+0.5 + 2 + 2 + 6},{2-0.5-2*sqrt(2) + 2*sqrt(2) + 1}) arc (180:135:1);
\draw[thick, smooth] ({2*sqrt(2)+2+0.5 + 2 + 4 + 6},{2-0.5-2*sqrt(2) + 2*sqrt(2) + 1}) arc (0:45:1);
\draw[thick, smooth]  ({2*sqrt(2)+2+0.5 + 2 + 2 + 6},{2-0.5-2*sqrt(2) + 2*sqrt(2) + 1})-- ({2*sqrt(2)+2+0.5 + 2 + 2 + 6},{2-1-2*sqrt(2) + 2*sqrt(2) + 1});
\draw[thick, smooth]  ({2*sqrt(2)+2+0.5 + 2 + 4 + 6},{2-0.5-2*sqrt(2) + 2*sqrt(2) + 1})-- ({2*sqrt(2)+2+0.5 + 2 + 4 + 6},{2-1-2*sqrt(2) + 2*sqrt(2) + 1});
\node at ({2*sqrt(2) + 2 + 0.5 + 10}, 0) {$\bullet$};
\node at ({2*sqrt(2) + 2 + 0.5 + 12}, 2) {$\bullet$};

\node[label={left: \fs ${\color{red} \shortrightarrow} \ x$}] at ({2*sqrt(2) + 2 + 0.5 + 8}, 0) {};
\node[label={left: \fs ${\color{red} \shortrightarrow} \ y$}] at ({2*sqrt(2) + 2 + 0.5 + 8}, 2) {};
\node[label={above: \fs ${\color{red} \begin{array}{c} \shortuparrow \\ {\color{black} \b{y}} \end{array} }$}] at ({2*sqrt(2)+2+0.5 + 2 + 2 + 6},{2 + 2*sqrt(2) + 1}) {};
\node[label={above: \fs ${\color{red} \begin{array}{c} \shortuparrow \\ {\color{black} \b{x}} \end{array} }$}] at ({2*sqrt(2)+2+0.5 + 2 + 4 + 6},{2 + 2*sqrt(2) + 1}) {};

\end{tikzpicture}
\caption{Reflection equation for corner vertices. The interpretation of this equation is analogous to that of Figure \ref{fig:refl}: the four external edges are assigned fixed arrows, while the internal edges are summed over. This produces $2^4$ equations. The rapidity variable associated to a line is reciprocated once the line crosses the dot $\bullet$ on a corner vertex.}
\label{fig:corner-reflect}
\end{figure} 
The off-diagonally symmetric partition function $Z_{\rm OSASM}$ depends on a single set of variables 
$\{x_1,\dots,x_{2n}\}$, of even cardinality. We give its graphical definition in Figure \ref{fig:osasm}. Using the regular Yang--Baxter equation in conjunction with the reflection equation for corner vertices, one can immediately deduce the symmetry of $Z_{\rm OSASM}$ in the variables $\{x_1,\dots,x_{2n}\}$.
\begin{figure}[H]
\begin{tikzpicture}[scale=0.6,>=stealth]
\foreach\x in {1,...,6}{
\draw[thick, smooth] (0,\x) -- (\x,\x);
\node at (0.5, \x) {\r};
}

\foreach\x in {1,...,6}
\node[label={left: \fs ${\color{red} \shortrightarrow} \ x_{\x}$}] at (0,7-\x) {};

\foreach\x in {1,...,6}{
\draw[thick]
(7-\x,7) -- (7-\x,7-\x);
\node at (\x,6.5) {\u};
}
\foreach\x in {1,...,6}
\node[label={above: \fs ${\color{red} \begin{array}{c}  \shortuparrow \\ {\color{black} \b{x}_{\x}} \end{array} }$}] at (7-\x,7) {};

\foreach\x in {1,...,6}
\node at (\x,\x) {$\bullet$};
\end{tikzpicture}
\caption{$Z_{\rm OSASM}$, the six-vertex model partition function on an off-diagonally symmetric lattice in the case $n=3$. We emphasize the fact that all vertical rapidities are reciprocated.}
\label{fig:osasm}
\end{figure}
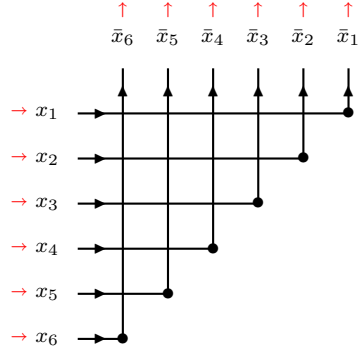

Kuperberg was able to evaluate $Z_{\rm OSASM}$ in closed form, as a Pfaffian \cite{kup2}:
\begin{align}
\label{kup-osasm}
Z_{\rm OSASM}(x_1,\dots,x_{2n};t)
=
\prod_{1 \leq i<j \leq 2n}
\frac{(1-t x_i x_j)}{(x_i-x_j)}
{\rm Pf}
\left[
\frac{(x_i-x_j)(1-t)}{(1-x_i x_j) (1-t x_i x_j)}
\right]_{1\leq i < j \leq 2n}.
\end{align}
Once again, this formula can be proved using an Izergin--Korepin type of approach, but for brevity we do not present these details here.

To draw a parallel with the analogous Sections \ref{sec:6v} and \ref{sec:U-vert}, we remark that $Z_{\rm OSASM}$ is a multiparameter generating series of the aforementioned OSASMs. It also appears, identically, on the right hand side of equation \eref{sym-pp-OSASM}. Hence \eref{sym-pp-OSASM} relates a 
path-weighted generating series for symmetric plane partitions with a generating series of OSASMs. This is the third such example relating these combinatorial objects (and their symmetry classes) that we have presented in this work. 

\subsection{Off-diagonal domains of odd size}
\label{ssec:p-osasm}

Thus far we only considered the case where the set of variables appearing in \eref{osasm-conj} has even cardinality. Clearly one can obtain a companion result for odd cardinalities, by specializing $x_{2n} = 0$ in \eref{osasm-conj}:
\newline

\noindent \textbf{Conjecture 2$\bm{'}$.}
\begin{multline}
\label{osasm-conj-pdwpf}
\sum_{\lambda'\ {\rm even}}
\
\prod_{j = 2,4,6,\dots}^{2n-\ell(\lambda)}
(1-t^{j-1})
\prod_{i=1}^{\infty}
\prod_{j = 2,4,6,\dots}^{m_i(\lambda)}
(1-t^{j-1})
P_{\lambda}(x_1,\dots,x_{2n-1};t)
=
\\
\frac{
(1-t)^n
\prod_{1 \leq i<j \leq 2n-1}
(1-t x_i x_j)
}
{
\prod_{i=1}^{2n-1}
(x_i)
\prod_{1 \leq i<j \leq 2n-1}
(x_i-x_j)
}
{\rm Pf}\Big[
\mathcal{O}_{i,j}
\Big]_{1 \leq i<j \leq 2n},  
\end{multline}
\textit{where the entries of the Pfaffian are given by}
\begin{align*}
\mathcal{O}_{i,j}
=
\frac{(x_i-x_j)}{(1-x_i x_j)(1-t x_i x_j)},
\ \ 
i<j<2n,
\quad\quad\quad\quad
\mathcal{O}_{i,j}
=
x_i,
\ \ 
i<j=2n.
\end{align*}
At the level of the six-vertex model, this specialization has an analogous interpretation to that discussed in Sections \ref{ssec:p-dwpf} and \ref{ssec:p-uasm}. One obtains a common total weight of $(1-t)$ for all possible configurations of the first column in Figure \ref{fig:osasm}, and we can simply delete this column from the lattice at the expense of this factor. The left external edges of the resulting odd-size lattice are summed over all possible arrow configurations. We illustrate this partition function in Figure \ref{fig:p-osasm}, and denote it by $Z_{\rm OSASM}(x_1,\dots,x_{2n-1};t)$. Cancelling a factor of $(1-t)$ from the right hand side of \eref{osasm-conj-pdwpf}, it is equal to $Z_{\rm OSASM}(x_1,\dots,x_{2n-1};t)$.

\begin{figure}[H]
\begin{tikzpicture}[scale=0.6,>=stealth]
\foreach\x in {1,...,5}{
\draw[thick, smooth] (0,\x) -- (\x,\x);
}

\foreach\x in {1,...,5}
\node[label={left: \fs ${\color{red} \shortrightarrow} \ x_{\x}$}] at (0,6-\x) {};

\foreach\x in {1,...,5}{
\draw[thick, smooth]
(6-\x,6) -- (6-\x,6-\x);
\node at (\x,5.5) {\u};
}
\foreach\x in {1,...,5}
\node[label={above: \fs ${\color{red} \begin{array}{c}  \shortuparrow \\ {\color{black} \b{x}_{\x}} \end{array} }$}] at (6-\x,6) {};

\foreach\x in {1,...,5}
\node at (\x,\x) {$\bullet$};
\end{tikzpicture}
\caption{$Z_{\rm OSASM}(x_1,\dots,x_{2n-1};t)$ in the case $n=3$, which results from setting $x_{2n}=0$ in Figure \ref{fig:osasm}. The left external edges are blank, to indicate that they are summed over all possible arrow configurations.}
\label{fig:p-osasm}
\end{figure}
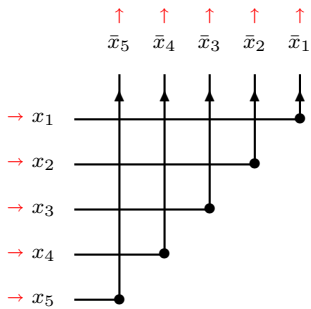

%
 

\section{Discussion}

In this paper we have discussed Cauchy and Littlewood identities, and simple refinements thereof, which can be evaluated in closed form. The expressions obtained on the right hand side of the new identities are partition functions of the six-vertex model, on suitable domains. We think that these results raise a number of questions and directions for further research, some of which we list below:   

{\bf 1.} It is interesting to observe some common structural features between the objects which equations \eref{s-uasm} and \eref{osasm-conj} relate. In the case of \eref{s-uasm}, the relationship is between symplectic plane partitions and the UASM partition function. The symplectic plane partitions have twice as many slices to the right of the main diagonal as to the left, with weightings that depend on $y_i$ and $\b{y}_i$ alternatingly. A similar structure is observed in the UASM partition function, which has twice as many horizontal lines as vertical lines, with alternating rapidities $x_i$ and $\b{x}_i$. In the case of \eref{osasm-conj}, the relationship is between symmetric plane partitions and the OSASM partition function. Both of these objects have a reflection symmetry about a central axis. While these similarities may be sheer coincidence, they might equally have a deeper explanation that could help in finding new correspondences between plane partitions and other symmetry classes of ASMs.

{\bf 2.} Since Theorem \ref{thm2} can be proved using the action of Macdonald's difference operators \eref{do} on the Cauchy identity \eref{hl-cauch2}, it would be aesthetically pleasing to obtain similar proofs of Conjectures \ref{conj1} and \ref{conj2} using appropriate difference operators. In the case of Conjecture \ref{conj2}, in particular, it is tempting to postulate the existence of difference operators\footnote{
Unfortunately, as is easily checked, acting on \eref{HL-little3} with the usual Macdonald difference operators \eref{do} {\it does not} give rise to the desired equation \eref{osasm-conj}.
} which act on the Littlewood identity \eref{HL-little3} to produce equation \eref{osasm-conj}. The situation seems more complicated in the case of Conjecture \ref{conj1}, since in that case we do not even know of a suitable Cauchy identity involving the $BC_n$-symmetric Hall--Littlewood polynomials, upon which we could act. It is possible that the Cauchy identity for $\tilde{K}_{\lambda}$ \eref{Ktilde-cauchy} and equation \eref{uasm-conj} are related by a suitable operation, but correctly identifying and classifying the properties of such difference operators is beyond the scope of this work.    

{\bf 3.} Are symplectic plane partitions worthy of further study in their own right, and do they have a 
well defined limit shape? Can one say anything about the associated particle processes and correlation functions? Moreover, in the $t$-deformed case, is there a (reasonable) branching rule for Rains' $\tilde{K}_{\lambda}$ polynomials and/or $BC_n$-symmetric Hall--Littlewood polynomials in terms of tableaux that would lead to plane partitions analogous to those of Vuleti\'c? Arguably here, providing an explicit formula for the branching rule for the $BC_n$-symmetric Hall--Littlewood polynomials would be an achievement in itself independent of the context.

\section*{Acknowledgments}

We would like to thank Andrea Sportiello for suggesting the idea of extending the identity \eref{knw-id} to ASM symmetry classes, which strongly motivated this work; Paul Zinn-Justin for explaining a technique for proving Theorem \ref{thm2} and Conjecture \ref{conj2}; and Eric Rains and Ole Warnaar for instructive discussions. This work was done under the support of the ERC grant 278124, ``Loop models, integrability and combinatorics''.

\appendix

\section{Analogue of Cauchy--Binet identity for Pfaffians}

The results presented in this appendix are taken from \cite{iw}. Let $T_{ij}$ denote the entries of an arbitrary $m \times M$ matrix, and $A_{ij}$ the entries of an $M\times M$ antisymmetric matrix, where $m \leq M$ and $m$ is even. Then
\begin{align*}
\sum_{\substack{
S \subseteq [M] \\ |S| = m
}}
{\rm Pf}[A_S] \det[T_S]
=
{\rm Pf}[T A T^{\rm t}]
=
{\rm Pf}
\left[
\sum_{1\leq k < l \leq M}
A_{kl}
\left|
\begin{array}{cc}
T_{ik} & T_{il}
\\
T_{jk} & T_{jl}
\end{array}
\right|
\right]_{1 \leq i<j \leq m}
\end{align*}
where the final identity follows from explicit calculation of the entries of the antisymmetric matrix 
$TAT^{\rm t}$. Changing notation slightly, we have 
\begin{align}
\sum_{1 \leq s_1 < \cdots < s_m \leq M}
{\rm Pf}
\left[
A_{s_i,s_j}
\right]_{1 \leq i<j \leq m}
\det
\left[
T_{i,s_j}
\right]_{1 \leq i,j \leq m}
=
{\rm Pf}
\left[
\sum_{1\leq k < l \leq M}
A_{kl}
(T_{ik} T_{jl} - T_{il} T_{jk})
\right]_{1 \leq i<j \leq m}.
\label{cb-analog1}
\end{align}
Of particular interest is the special case $T_{ij} = x_i^{j-1}$, when equation \eref{cb-analog1} becomes
\begin{align}
\sum_{1 \leq s_1 < \cdots < s_m \leq M}
{\rm Pf}
\left[
A_{s_i,s_j}
\right]_{1 \leq i<j \leq m}
\det
\left[
x_i^{s_j-1}
\right]_{1 \leq i,j \leq m}
&=
{\rm Pf}
\left[
\sum_{1 \leq k < l \leq M}
A_{kl} (x_i^{k-1} x_j^{l-1} - x_i^{l-1} x_j^{k-1})
\right]_{1 \leq i < j \leq m}.
\label{cb-analog2}
\end{align}
In principle, equation \eref{cb-analog2} allows a Pfaffian of a general bilinear function in $x_i$ and $x_j$ to be expanded in the basis of Schur polynomials.

\bibliographystyle{abbrv}
\bibliography{PPASM}

\providecommand{\noopsort}[1]{}
\begin{thebibliography}{10}

\bibitem{bax}
R.~J. Baxter.
\newblock {\em Exactly solved models in statistical mechanics}.
\newblock Academic Press Inc. [Harcourt Brace Jovanovich Publishers], London,
  1982.

\bibitem{bor}
A.~Borodin.
\newblock Schur dynamics of the {S}chur processes.
\newblock {\em Adv. Math.}, 228(4):2268--2291, 2011.

\bibitem{vde}
{\noopsort{Diejen} J. F. van Diejen} and E.~Emsiz.
\newblock The semi-infinite {$q$}-boson system with boundary interaction.
\newblock {\em Lett. Math. Phys.}, 104(1):103--113, 2014.

\bibitem{fw}
O.~Foda and M.~Wheeler.
\newblock Hall-{L}ittlewood plane partitions and {KP}.
\newblock {\em Int. Math. Res. Not. IMRN}, (14):2597--2619, 2009.

\bibitem{hk}
A.~M. Hamel and R.~C. King.
\newblock U-turn alternating sign matrices, symplectic shifted tableaux and
  their weighted enumeration.
\newblock {\em J. Algebraic Combin.}, 21(4):395--421, 2005.

\bibitem{iw}
M.~Ishikawa and M.~Wakayama.
\newblock Applications of minor-summation formula. {II}. {P}faffians and
  {S}chur polynomials.
\newblock {\em J. Combin. Theory Ser. A}, 88(1):136--157, 1999.

\bibitem{ize}
A.~G. Izergin.
\newblock Partition function of the six-vertex model in a finite volume.
\newblock {\em Dokl. Akad. Nauk SSSR}, 297(2):331--333, 1987.

\bibitem{jm}
M.~Jimbo and T.~Miwa.
\newblock Solitons and infinite-dimensional {L}ie algebras.
\newblock {\em Publ. Res. Inst. Math. Sci.}, 19(3):943--1001, 1983.

\bibitem{kes}
R.~C. King and N.~G.~I. El-Sharkaway.
\newblock Standard {Y}oung tableaux and weight multiplicities of the classical
  {L}ie groups.
\newblock {\em J. Phys. A}, 16(14):3153--3177, 1983.

\bibitem{kn}
A.~N. Kirillov and M.~Noumi.
\newblock {$q$}-difference raising operators for {M}acdonald polynomials and
  the integrality of transition coefficients.
\newblock In {\em Algebraic methods and {$q$}-special functions ({M}ontr\'eal,
  {QC}, 1996)}, volume~22 of {\em CRM Proc. Lecture Notes}, pages 227--243.
  Amer. Math. Soc., Providence, RI, 1999.

\bibitem{kor}
V.~E. Korepin.
\newblock Calculation of norms of {B}ethe wave functions.
\newblock {\em Comm. Math. Phys.}, 86(3):391--418, 1982.

\bibitem{kbi}
V.~E. Korepin, N.~M. Bogoliubov, and A.~G. Izergin.
\newblock {\em Quantum inverse scattering method and correlation functions}.
\newblock Cambridge Monographs on Mathematical Physics. Cambridge University
  Press, Cambridge, 1993.

\bibitem{kup1}
G.~Kuperberg.
\newblock Another proof of the alternating-sign matrix conjecture.
\newblock {\em Internat. Math. Res. Notices}, (3):139--150, 1996.

\bibitem{kup2}
G.~Kuperberg.
\newblock Symmetry classes of alternating-sign matrices under one roof.
\newblock {\em Ann. of Math. (2)}, 156(3):835--866, 2002.

\bibitem{llt}
D.~Laksov, A.~Lascoux, and A.~Thorup.
\newblock On {G}iambelli's theorem on complete correlations.
\newblock {\em Acta Math.}, 162(3-4):143--199, 1989.

\bibitem{mac}
I.~G. Macdonald.
\newblock {\em Symmetric functions and {H}all polynomials}.
\newblock Oxford Mathematical Monographs. The Clarendon Press Oxford University
  Press, New York, second edition, 1995.
\newblock With contributions by A. Zelevinsky, Oxford Science Publications.

\bibitem{macm}
P.~A. MacMahon.
\newblock Memoir on the theory of the partitions of numbers {VI}: partitions in
  two-dimensional space, to which is added an adumbration of the theory of
  partitions in three-dimensional space.
\newblock {\em Phil. Trans. Roy. Soc. London Ser. A}, 211:345--373, 1912.

\bibitem{mrr}
W.~H. Mills, D.~P. Robbins, and H.~Rumsey, Jr.
\newblock Alternating sign matrices and descending plane partitions.
\newblock {\em J. Combin. Theory Ser. A}, 34(3):340--359, 1983.

\bibitem{oka}
S.~Okada.
\newblock Enumeration of symmetry classes of alternating sign matrices and
  characters of classical groups.
\newblock {\em J. Algebraic Combin.}, 23(1):43--69, 2006.

\bibitem{oko}
A.~Okounkov.
\newblock Infinite wedge and random partitions.
\newblock {\em Selecta Math. (N.S.)}, 7(1):57--81, 2001.

\bibitem{or}
A.~Okounkov and N.~Reshetikhin.
\newblock Correlation function of {S}chur process with application to local
  geometry of a random 3-dimensional {Y}oung diagram.
\newblock {\em J. Amer. Math. Soc.}, 16(3):581--603 (electronic), 2003.

\bibitem{rai}
E.~M. Rains.
\newblock {${\rm BC}_n$}-symmetric polynomials.
\newblock {\em Transform. Groups}, 10(1):63--132, 2005.

\bibitem{skl}
E.~K. Sklyanin.
\newblock Boundary conditions for integrable quantum systems.
\newblock {\em J. Phys. A}, 21(10):2375--2389, 1988.

\bibitem{sta}
R.~P. Stanley.
\newblock {\em Enumerative combinatorics. {V}ol. 2}, volume~62 of {\em
  Cambridge Studies in Advanced Mathematics}.
\newblock Cambridge University Press, Cambridge, 1999.
\newblock With a foreword by Gian-Carlo Rota and appendix 1 by Sergey Fomin.

\bibitem{ste}
J.~R. Stembridge.
\newblock Nonintersecting paths, {P}faffians, and plane partitions.
\newblock {\em Adv. Math.}, 83(1):96--131, 1990.

\bibitem{sun}
S.~Sundaram.
\newblock Tableaux in the representation theory of the classical {L}ie groups.
\newblock In {\em Invariant theory and tableaux ({M}inneapolis, {MN}, 1988)},
  volume~19 of {\em IMA Vol. Math. Appl.}, pages 191--225. Springer, New York,
  1990.

\bibitem{tsi}
N.~V. Tsilevich.
\newblock The quantum inverse scattering problem method for the {$q$}-boson
  model, and symmetric functions.
\newblock {\em Funktsional. Anal. i Prilozhen.}, 40(3):53--65, 96, 2006.

\bibitem{tsu}
O.~Tsuchiya.
\newblock Determinant formula for the six-vertex model with reflecting end.
\newblock {\em J. Math. Phys.}, 39(11):5946--5951, 1998.

\bibitem{ven}
V.~Venkateswaran.
\newblock Symmetric and nonsymmetric {H}all-{L}ittlewood polynomials of type
  {BC}.
\newblock {\em arXiv:1209.2933v2}, 2013.

\bibitem{vul}
M.~Vuleti{\'c}.
\newblock A generalization of {M}ac{M}ahon's formula.
\newblock {\em Trans. Amer. Math. Soc.}, 361(5):2789--2804, 2009.

\bibitem{war}
S.~O. Warnaar.
\newblock Bisymmetric functions, {M}acdonald polynomials and {$sl_3$} basic
  hypergeometric series.
\newblock {\em Compos. Math.}, 144(2):271--303, 2008.

\end{thebibliography}

\end{document}